\documentclass[12pt]{article}
\usepackage{aurical}
\usepackage[T1]{fontenc}
\usepackage{amssymb,amsmath,bm}
\usepackage{amsthm}
\usepackage{setspace,mathrsfs,nicefrac}
\usepackage{bbm,url,color,wasysym}
\usepackage[margin=1in]{geometry}
\usepackage{paralist}\setdefaultitem{$-$}{\tiny $\bullet$}{\tiny$\circ$}{$\star}
\usepackage{hyperref}
\usepackage[title]{appendix}
\usepackage{tikz,pgf}
\usepackage{color}
\allowdisplaybreaks

\renewcommand{\geq}{\geqslant}
\renewcommand{\leq}{\leqslant}

\author{{\bf Robert C. Dalang$^*$}  and {\bf Fei Pu\footnote{This paper is based on F. Pu's Ph.D. thesis, written under the supervision of R. C. Dalang. Research
partially supported by the Swiss National Foundation for Scientific
Research.}}\\
\\
\it\small \'Ecole Polytechnique F\'ed\'erale de Lausanne \\
}

\title{\bf\Large{On the density of the supremum of the solution to the linear stochastic heat equation}}
\date{}

\newtheorem{theorem}{Theorem}[section]
\newtheorem{cor}[theorem]{Corollary}
\newtheorem{lemma}[theorem]{Lemma}
\newtheorem{prop}[theorem]{Proposition}
\newtheorem{remark}[theorem]{Remark}
\numberwithin{equation}{section}

\begin{document}
\maketitle
\begin{abstract}
We study the regularity of the probability density function of the supremum of the solution to the linear stochastic heat equation. Using a general criterion for the smoothness of densities for locally nondegenerate random variables, we establish the smoothness of the joint density of the random vector whose components are the solution and the supremum of an increment in time of the solution over an interval (at a fixed spatial position), and the smoothness of the density of the supremum of the solution over a space-time rectangle that touches the $t = 0$ axis. Applying the properties of the divergence operator, we establish a Gaussian-type upper bound on these two densities respectively, which presents a close connection with the H\"{o}lder-continuity properties of the solution.
\end{abstract}
\vskip.5in 



\noindent{\it Abbreviated title: Density of supremum for linear SPDE}

\section{Introduction and main results} \label{section4.1}

For a real-valued Gaussian random field $\{X(t): t \in I\}$, where $I$ is a parameter set, the distribution function of the supremum of this random field, or the excursion probability $\mbox{P}\{\sup_{t \in I}X(t) \geq a\}$, has been investigated extensively; see, for example, \cite{Adl90, AdT07, Pit96} and references therein.
In general, finding a formula for the distribution function of the supremum of a stochastic process is an almost impossible task, let alone for its probability density function, which is much less studied than the probability distribution function.

 The question of smoothness of the density of the supremum of a multiparameter Gaussian process dates back to the work of Florit and Nualart \cite{FlN95}, in which they establish a general criterion (see Theorem \ref{th2017-08-24-2}) for the smoothness of the density, assuming that the random vector is locally in $\mathbb{D}^{\infty}$, and apply it to show that the maximum of the Brownian sheet on a rectangle possesses an infinitely differentiable density. Moreover, this method was applied to prove that the supremum of a fractional Brownian motion has an infinitely differentiable density; see Lanjri Zadi and Nualart \cite{LaN03}. Nourdin and Viens \cite{NoV09} use the Malliavin calculus to obtain a formula for the density of the law of any random variable which is measurable and differentiable with respect to a given isonormal Gaussian process and they apply this result to study the density of supremum (after centering) of a Gaussian process whose covariance function is bounded below and above by positive constants; see also \cite[Section 5.5]{Nua18} for a discussion on this density formula.
Some general results on the regularity of the density of the maximum of Gaussian random fields have been developed by Cirel'son \cite{Cir75}, Pitt and Lanh \cite{PiL79}, Weber \cite{Web85}, Lifshits \cite{Lif86, Lif95}, Diebolt and Posse \cite{DiP96} and Aza\"{\i}s and Wschebor \cite{AzW09}.
We also refer to Hayashi and Kohatsu-Higa \cite{HaK13} and Nakatsu \cite{Nak16} for the smoothness of densities for the supremum of diffusion processes.

We are interested in the properties of the probability density function of the supremum of the solutions to SPDEs. This is partly motivated by the fact that the density of the supremum  of the solution is related to the study of upper bounds on hitting probabilities for these solutions; see \cite[Section 4.1]{Pu18} for detailed explanation.

We consider the linear stochastic heat equation
 \begin{align}\label{eq2017-10-11-62}
 \frac{\partial u}{\partial t}(t, x) = \frac{\partial^2u}{\partial x^2}(t, x) + \dot{W}(t, x),
 \end{align}
 for $t \in [0, \infty[$ and $x \in [0, 1]$, with initial condition $u(0, x) = 0$ for all $x \in [0, 1]$, and either Neumann or Dirichlet boundary conditions. Here, $W = \{W(t, x): (t, x) \in [0, \infty[ \times [0, 1]\}$ denotes Brownian sheet defined on a probability space $(\Omega, \mathscr{F}, \mbox{P})$.

By definition \cite{Wal86}, the mild solution is
  \begin{align}\label{eq2}
  u(t, x) = \int_0^t\int_0^1 G(t - r, x, v) W(dr, dv),
 \end{align}
 where $G(t, x, y)$ is the Green kernel for the heat equation with Neumann or Dirichlet boundary conditions; see  \cite{BMS95}.

We assume that the process $\{u(t, x): (t, x) \in [0, T] \times [0, 1]\}$ given by (\ref{eq2}) is the jointly continuous version, which is almost $\frac{1}{4}$-H\"{o}lder continuous in time and almost $\frac{1}{2}$-H\"{o}lder continuous in space. In fact, for any $p \geq 1, (t, x), (s, y) \in [0, T] \times [0, 1]$, there exists a constant $C = C(p, T)$ such that
\begin{align}\label{eq2016-06-14-1}
\mbox{E}[|u(t, x) - u(s, y)|^p] \leq C (|t - s|^{1/2} + |x - y|)^{p/2}.
\end{align}

Choose two non-trivial compact intervals $I \subset [0, T]$ and $J \subset [0, 1]$. In the case of Dirichlet boundary conditions, we assume that $J \subset \, ]0, 1[$.  Choose $\delta_1 > 0$ and $(s_0, y_0) \in I \times J$. For $t\in [0, T]$, we denote
\begin{align} \label{eq2018-03-01-600}
\bar{u}(t, y_0) =  u(t, y_0) - u(s_0, y_0).
\end{align}
Set
\begin{align} \label{eq2017-12-19-2}
F_1 = u(s_0, y_0), \quad F_2 = \sup_{
 t \in [s_0, s_0 + \delta_1]
}\bar{u}(t, y_0) \quad \mbox{and}\quad F = (F_1, F_2).
 \end{align}
Choose $\delta_2 > 0$ such that $[y_0, y_0 + \delta_2] \subset [0, 1]$; in the case of Dirichlet boundary conditions, we assume that $[y_0, y_0 + \delta_2]  \subset \, ]0, 1[$ (open interval). Denote by $M_0$ the global supremum of $u$ over $[0, \delta_1] \times [y_0, y_0 + \delta_2]$:
\begin{align}\label{eq2018-01-02-4}
M_0 = \sup_{
 (t,x) \in [0, \delta_1] \times [y_0, y_0 + \delta_2]
 }u(t, x).
\end{align}

Our goal is to give estimates on the joint probability density function of $F$, and on the probability density function of $M_0$. At the same time, we also want to determine if these two probability density functions are infinitely differentiable. Malliavin calculus is a tool to study the smoothness of random variables (see \cite[Proposition 2.1.5]{Nua06} or \cite[Theorem 5.2]{San05}). It is clear that the first component of $F$ belongs to $\mathbb{D}^{\infty}$. We will show that $F_2$ belongs to $\mathbb{D}^{1, 2}$ in Lemma \ref{lemma9}. However, we do not expect that $F_2$ belongs to $\mathbb{D}^{\infty}$. The same problem arises for $M_0$. This means we can not apply standard results for {\em nondegenerate} random variables.

Florit and Nualart \cite{FlN95} established a general criterion (see Theorem \ref{th2017-08-24-2} below) for smoothness of the density assuming only that the components of the random vector belong to $\mathbb{D}^{1, 2}$. We will make use of this result to prove the smoothness of the densities of the random variables $F$ and $M_0$.

We first state our results on the smoothness of the densities of these random variables.
\begin{theorem}\label{the2017-11-25-1}
\begin{itemize}
  \item [(a)]For all $(s_0, y_0) \in \, ]0, T] \times J$ and $\delta_1 > 0$, the random vector $F$  has an infinitely differentiable density on $\mathbb{R} \times ]0, \infty[$ and if $y_0 \in \, ]0, 1[$, then $F_2 > 0$ a.s.  When $s_0 = 0$, $F_1$ vanishes identically but $F_2$ takes values in $]0, \infty[$ a.s. and has an infinitely differentiable density on $]0, \infty[$.
     \item [(b)]For all $y_0 \in [0, 1]$, $\delta_1 > 0$ and $\delta_2 > 0$ with $[y_0, y_0 + \delta_2] \subset [0, 1]$ ($[y_0, y_0 + \delta_2] \subset \, ]0, 1[$ in the case of Dirichlet boundary conditions), the random variable $M_0$ takes values in $]0, \infty[$ and has an infinitely differentiable density on $]0, \infty[$.
  \end{itemize}
\end{theorem}
\noindent{}Theorem \ref{the2017-11-25-1} will be proved in Section \ref{section4.3}.

 By inspecting the proof of Theorem \ref{th2017-08-24-2} (see \cite[Theorem 2.1.4]{Nua06}), we can observe that the integration by parts formula leads us to a new formula for the density of the random vector $F$ (see Proposition \ref{prop2017-10-02-1}). From this formula, it becomes possible to analyze the behavior of the density. Since the choice of the random vector $u_A$ in Theorem \ref{th2017-08-24-2} is not unique, we will see that it is possible to choose this $u_A$ so that it is an adapted process for which the Skorohod integral coincides with the Walsh integral. This makes it possible to use Burkholder's inequality, instead of H\"{o}lder's inequality for Malliavin norms (see \cite[Proposition 1.10, p.50]{Wat84}) to estimate the moments of this stochastic integral. This will allow us to give a sharp Gaussian-type upper bound on this density.

In order to estimate the density of $F$, we assume $I \times J \subset \, ]0, T] \times ]0, 1[$. Assume that there are constants $c_1, C_1$ such that
\begin{align} \label{eq2017-12-26-1}
0 < c_1 < \underline{I} := \inf\{s: s \in I\} \quad  \mbox{and} \quad \bar{I} := \sup\{s: s \in I\} < C_1 < T + 1.
\end{align}
Assume also that there are constants $c_2, C_2$ such that
\begin{align}\label{eq2017-12-26-3}
0 < c_2 < \underline{J} := \inf\{y: y \in J\} \quad  \mbox{and} \quad  \bar{J} := \sup\{y: y \in J\} < C_2 < 1.
\end{align}
Choose $\delta_1 \in \, ]0, 1[$ small enough so that
\begin{align}\label{eq2017-09-26}
s_0 + \delta_1 \in I  \quad \mbox{and} \quad   \delta_1^{1/2} < \min\left\{\underline{J} - c_2, (C_2 - \bar{J})/2\right\};
\end{align}
see Figure \ref{figure1}.

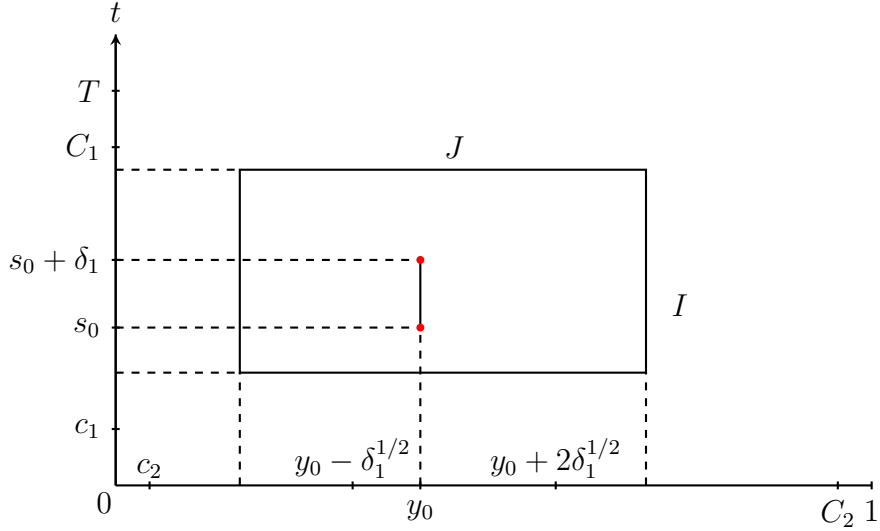
\begin{figure}

\centering {
\begin{tikzpicture}[>=stealth,scale=1.5,line width=0.8pt]


\pgfmathsetmacro{\ticker}{0.125} 

\coordinate (A) at (0,0);
\draw (-0.1,-0.15) node {$0$};
\coordinate (B) at (0,4);
\coordinate (D) at (6.7,0);
\draw(A)--(B)--cycle;
\draw(A)--(D)--cycle;
\draw[->] (0, 0) -- (0, 4);

\draw (1.1,1) rectangle (4.7, 2.8);
\draw (0,4.2) node { $\large{t}$};

\draw[dashed] (0, 1) -- (1.1, 1);
\draw[dashed] (1.1, 0) -- (1.1, 1);
\draw[dashed] (0, 2.8) -- (1.1, 2.8);
\draw[dashed] (2.7, 0) -- (2.7, 1.4);

\draw[dashed] (4.7, 0) -- (4.7, 1);
\draw[dashed] (0, 1.4) -- (2.7, 1.4);
\draw[dashed] (0, 2.0) -- (2.7, 2.0);
\draw (2.7,1.4) rectangle (2.7, 2.0);

\fill[red](2.7,1.4) circle (1pt);

\fill[red](2.7,2.0) circle (1pt);

\draw (2.1, 1pt) -- (2.1, -1pt) node[anchor=south] {$y_0 - \delta_1^{1/2} $};

\draw (2.7, 1pt) -- (2.7, -1pt) node[anchor=north] {$y_0$};

\draw (3.9, 1pt) -- (3.9, -1pt) node[anchor=south] {$y_0 + 2\delta_1^{1/2}$};
\draw (0.3, 1pt) -- (0.3, -1pt) node[anchor=south] {$c_2$};
\draw (6.4, 1pt) -- (6.4, -1pt) node[anchor=north] {$C_2$};
\draw (6.7, 1pt) -- (6.7, -1pt) node[anchor=north] {$1$};

{\draw (1pt,3.5) -- (-1pt,3.5) node[anchor=east] {$T$};}
{\draw (1pt,0.5) -- (-1pt,0.5) node[anchor=east] {$c_1$};}
{\draw (1pt,3) -- (-1pt,3) node[anchor=east] {$C_1$};}

{\draw (1pt,1.4) -- (-1pt,1.4) node[anchor=east] {$s_0$};}
{\draw (1pt,2.0) -- (-1pt,2.0) node[anchor=east] {$s_0 + \delta_1$};}

\draw (5,1.6) node {$I$};
\draw (3.,3) node {$J$};

\end{tikzpicture}
}%
\caption{Illustration of conditions \eqref{eq2017-12-26-1}--\eqref{eq2017-09-26}}\label{figure1}
\end{figure}

Denote $(z_1, z_2) \mapsto p(z_1, z_2)$ the probability density function of random vector $F$ with $\delta_1$ satisfying the conditions in \eqref{eq2017-09-26} (the existence of $p(\cdot, \cdot)$ is assured by Theorem \ref{the2017-11-25-1}(a)).

\begin{theorem}\label{theorem100}
Assume $I \times J \subset \, ]0, T] \times ]0, 1[$. There exists a positive constant $c = c(I, J)$ such that for all $\delta_1 > 0$ satisfying \eqref{eq2017-09-26}, and for all $z_2 \geq \delta_1^{1/4}$, $z_1 \in \mathbb{R}$ and any $(s_0, y_0) \in I \times J$,
\begin{align}
p(z_1, z_2) &\leq \frac{c}{\delta_1^{1/4}}\exp\left(-\frac{z_2^2}{c\, \delta_1^{1/2}}\right)(|z_1|^{-1/4} \wedge 1)\exp(-z_1^2/c)\label{eq2017-09-26-1}\\
   &\leq \frac{c}{\delta_1^{1/4}}\exp\left(-\frac{z_2^2}{c\, \delta_1^{1/2}}\right). \label{eq2017-09-08-100}
\end{align}
\end{theorem}
\noindent{}The proof of this theorem will be presented in  Section \ref{section4.4}.
Note that (\ref{eq2017-09-08-100}) is an immediate consequence of (\ref{eq2017-09-26-1}). As a consequence of Theorems \ref{the2017-11-25-1} and \ref{theorem100}, we deduce the following.
\begin{cor}\label{th2017-08-29-100}
Let $I$ and $J$  be as above \eqref{eq2018-03-01-600}. The random variable $F_2$ has an infinitely differentiable density on $]0, \infty[$, denoted by $z_2 \mapsto p_{F_2}(z_2)$ . Suppose that $I \times J \subset \, ]0, T] \times ]0, 1[$. Then there exists a positive constant $c = c(I, J)$ such that for all $\delta_1 > 0$ satisfying \eqref{eq2017-09-26}, and for all $z_2 \geq \delta_1^{1/4}$ and any $(s_0, y_0) \in I \times J$,
\begin{align*}
p_{F_2}( z_2) \leq \frac{c}{\delta_1^{1/4}}\exp\left(-\frac{z_2^2}{c\, \delta_1^{1/2}}\right).
\end{align*}
\end{cor}

\begin{remark}
Theorem \ref{theorem100} provides an alternative proof of \cite[Theorem 3.1(3)]{DKN07} with $\beta = d$ there for the upper bound on hitting probabilities at a fixed spatial position; see \cite[Remark 4.2.4]{Pu18} for the explanation.
\end{remark}

We will also give a Gaussian-type upper bound on the density of $M_0$ under the assumption $y_0 \in J \subset \, ]0, 1[$. Choose a positive constant $\bar{C}_1$ with $\bar{C}_1 < T$. Let $c_2$, $C_2$ be chosen as in \eqref{eq2017-12-26-3}. Choose $\delta_1$, $\delta_2 \in \, ]0, 1[$ small enough so that
\begin{align}\label{eq2017-09-2600}
y_0 + \delta_2 \in J, \quad (\delta_1^{1/2} + \delta_2)^2 < \bar{C}_1   \quad \mbox{and} \quad  \delta_1^{1/2} + \delta_2< \min\left\{\underline{J} - c_2, (C_2 - \bar{J})/2\right\};
\end{align}
see Figure \ref{figure2018-01-03-1}.

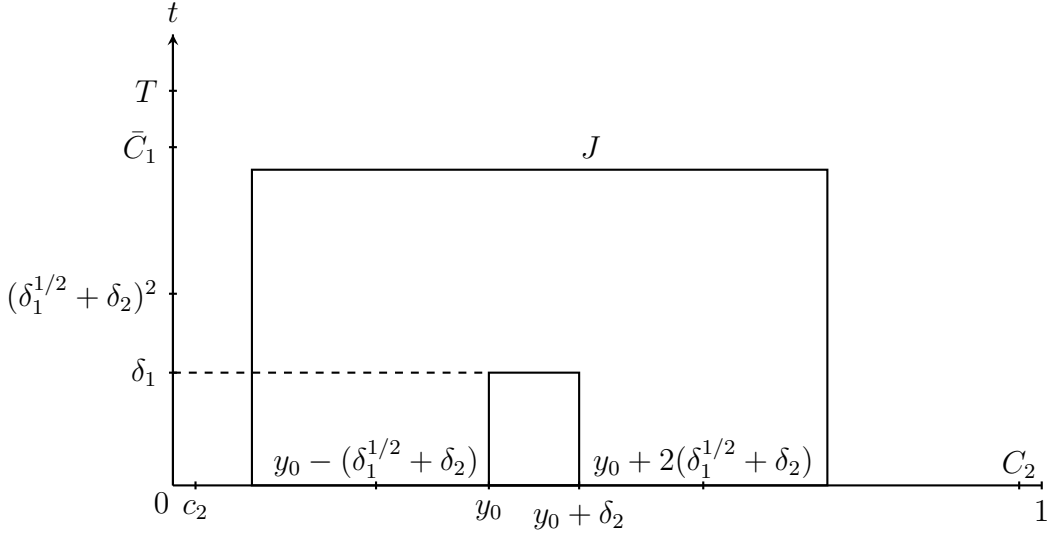
\begin{figure}

\centering {
\begin{tikzpicture}[>=stealth,scale=1.5,line width=0.8pt]


\pgfmathsetmacro{\ticker}{0.125} 

\coordinate (A) at (0,0);
\draw (-0.1,-0.15) node {$0$};
\coordinate (B) at (0,4);
\coordinate (D) at (7.7,0);
\draw(A)--(B)--cycle;
\draw(A)--(D)--cycle;
\draw[->] (0, 0) -- (0, 4);

\draw (0.7,0) rectangle (5.8, 2.8);
\draw (0,4.2) node { $\large{t}$};
\draw[dashed] (0, 1) -- (2.8, 1);
\draw (2.8,0) rectangle (3.6, 1);

\draw (1.8, 1pt) -- (1.8, -1pt) node[anchor=south] {$y_0 - (\delta_1^{1/2} + \delta_2)$};

\draw (2.8, 1pt) -- (2.8, -1pt) node[anchor=north] {$y_0$};

\draw (3.6,1pt) -- (3.6, -1pt) node[anchor=north] {$y_0 + \delta_2$};

\draw (4.7, 1pt) -- (4.7, -1pt) node[anchor=south] {$y_0 + 2(\delta_1^{1/2} + \delta_2)$};
\draw (0.2, 1pt) -- (0.2, -1pt) node[anchor=north] {$c_2$};
\draw (7.5, 1pt) -- (7.5, -1pt) node[anchor=south] {$C_2$};
\draw (7.7, 1pt) -- (7.7, -1pt) node[anchor=north] {$1$};

{\draw (1pt,3.5) -- (-1pt,3.5) node[anchor=east] {$T$};}

{\draw (1pt,3) -- (-1pt,3) node[anchor=east] {$\bar{C}_1$};}
{\draw (1pt,1) -- (-1pt,1) node[anchor=east] {$\delta_1$};}
{\draw (1pt,1.7) -- (-1pt,1.7) node[anchor=east] {$(\delta_1^{1/2} + \delta_2)^2$};}
\draw (3.7,3) node {$J$};
\end{tikzpicture}
}%
\caption{Illustration of condition \eqref{eq2017-09-2600}}\label{figure2018-01-03-1}
\end{figure}

Denote $z \mapsto p_0(z)$ the probability density function of random variable $M_0$ with $\delta_1$, $\delta_2$ satisfying the conditions in \eqref{eq2017-09-2600} (the existence of $p_0(\cdot)$ is assured by Theorem \ref{the2017-11-25-1}(b)).

\begin{theorem} \label{theorem2018-01-03-1}
Assume $J \subset \, ]0, 1[$. There exists a finite positive constant $c = c(T, J)$ such that for all $\delta_1$, $\delta_2$ satisfying the conditions in \eqref{eq2017-09-2600}, for all $y_0 \in J$  and $z \geq (\delta_1^{1/2} + \delta_2)^{1/2}$,
                     \begin{align}\label{eq2018-01-02-16}
                     p_0(z) \leq \frac{c}{\sqrt{\delta_1^{1/2} + \delta_2}}\exp\left(-\frac{z^2}{c \, (\delta_1^{1/2} + \delta_2)}\right).
                     \end{align}

\end{theorem}
\noindent{}The proof of Theorem \ref{theorem2018-01-03-1} will be presented in  Section \ref{section4.7}.

\begin{remark}
It is also interesting to consider the supremum over a spatial interval at a fixed time, such as
\begin{align}\label{eq2018-10-18-1}
\sup_{x \in [y_0, y_0 + \delta_2]}(u(s_0, x) - u(s_0, y_0))
\end{align}
and the supremum over a general rectangle such as
\begin{align*}
\sup_{(t,x) \in [s_0, s_0 + \delta_1] \times [y_0, y_0 + \delta_2]}(u(t, x) - u(s_0, y_0)).
\end{align*}
These questions are somewhat different (for instance, there is no time evolution in \eqref{eq2018-10-18-1}) and are the subjects of current research.
\end{remark}

\section{Elements of Malliavin calculus}

In this section, we introduce, following Nualart \cite{Nua06} (see also \cite{San05}), some elements of Malliavin calculus. Let $W = \{W(h), h \in \mathscr{H}\}$ denote the isonormal Gaussian process (see \cite[Definition 1.1.1]{Nua06}) associated with space-time white noise, where $\mathscr{H}$ is the Hilbert space $L^2([0, T] \times [0, 1], \mathbb{R})$. Let $\mathscr{S}$ denote the class of smooth random variables of the form
\begin{align*}
G = g(W(h_1), \ldots, W(h_n)),
\end{align*}
where $n \geq 1$, $g \in \mathscr{C}_p^{\infty}(\mathbb{R}^n)$, the set of real-valued functions $g$ such that $g$ and all its partial derivatives have at most polynomial growth and $h_i \in \mathscr{H}$. Given $G \in \mathscr{S}$, its derivative is defined to be the $\mathbb{R}$-valued stochastic process $DG = (D_{t, x}G, \, (t, x) \in [0, T] \times [0, 1])$ given by
\begin{align*}
D_{t, x}G = \sum_{i = 1}^n \partial_i g (W(h_1), \ldots, W(h_n))h_i(t, x).
\end{align*}
More generally, we can define the derivative $D^kG$ of order $k$ of $G$ by setting
 \begin{align*}
D_{\alpha}^kG = \sum_{i_1, \ldots, i_k = 1}^n \frac{\partial}{\partial x_{i_1}}\cdots \frac{\partial}{\partial x_{i_k}}\, g(W(h_1), \ldots, W(h_n))h_{i_1}(\alpha_1)\otimes \cdots \otimes h_{i_k}(\alpha_k),
\end{align*}
where $\alpha = (\alpha_1, \ldots, \alpha_k)$, $\alpha_i = (t_i, x_i), 1 \leq i \leq k$ and the notation $\otimes$ denotes the tensor product of functions.

For $p$, $k \geq 1$, the space $\mathbb{D}^{k, p}$ is the closure of $\mathscr{S}$ with respect to the seminorm $\|\cdot\|_{k, p}$ defined by
\begin{align*}
\|G\|_{k, p}^p = \mbox{E}[|G|^p] + \sum_{j = 1}^k\mbox{E}\left[\|D^jG\|_{\mathscr{H}^{\otimes j}}^p\right],
\end{align*}
where
\begin{align*}
\|D^jG\|_{\mathscr{H}^{\otimes j}}^2 = \int_0^Tdt_1\int_0^1dx_1 \cdots \int_0^Tdt_j\int_0^1dx_j\left(D_{t_1, x_1} \cdots D_{t_j, x_j}G\right)^2.
\end{align*}
We set $\mathbb{D}^{\infty} = \cap_{p \geq 1}\cap_{k \geq 1}\mathbb{D}^{k, p}$.

For any given Hilbert space $V$, the corresponding Sobolev space of $V$-valued random variables can also be introduced. More precisely, let $\mathscr{S}_V$ denote the family of $V$-valued smooth random variables of the form
$
G = \sum_{j = 1}^nG_jv_j, \, (v_j, G_j) \in V \times \mathscr{S}.
$
We define
\begin{align*}
D^kG = \sum_{j = 1}^n(D^kG_j) \otimes v_j, \quad k \geq 1.
\end{align*}
Then $D^k$ is a closable operator from $\mathscr{S}_V \subset L^p(\Omega, V)$ into $L^p(\Omega, \mathscr{H}^{\otimes k} \otimes V)$ for any $p \geq 1$. For $p, k \geq 1$, a seminorm $\|\cdot\|_{k, p, V}$ is defined on $\mathscr{S}_V$ by
\begin{align*}
\|G\|_{k, p, V}^p = \mbox{E}\left[\|G\|_{V}^p\right] + \sum_{j = 1}^k \mbox{E}\left[\|D^jG\|^p_{\mathscr{H}^{\otimes j} \otimes V}\right].
\end{align*}
We denote by $\mathbb{D}^{k, p}(V)$ the closure of $\mathscr{S}_V$ with respect to the seminorm $\|\cdot\|_{k, p, V}$. We set $\mathbb{D}^{\infty}(V) = \cap_{p \geq 1}\cap_{k \geq 1}\mathbb{D}^{k, p}(V)$.

The derivative operator $D$ on $L^2(\Omega)$ has an adjoint, termed the Skorohod integral and denoted by $\delta$, which is an unbounded and closed operator on $L^2(\Omega, \mathscr{H})$; see \cite[Section 1.3]{Nua06}. Its domain, denoted by  $\mbox{Dom}\ \delta$,  is the set of elements $u \in L^2(\Omega, \mathscr{H})$ such that there exists a constant $c$ such that $|\mbox{E}[\langle DG, u \rangle_{\mathscr{H}}]| \leq c \|G\|_{0, 2}$, for any $G \in \mathbb{D}^{1, 2}$. If $u \in \mbox{Dom}\ \delta$, then $\delta(u)$ is the element of $L^2(\Omega)$ characterized by the following duality relation:
\begin{align*}
\mbox{E}[G\, \delta(u)] = \mbox{E}\left[\int_0^T\int_0^1 D_{t, x}G \, u(t, x)dtdx\right], \quad \mbox{for all }\ G \in \mathbb{D}^{1, 2}.
\end{align*}

In order to handle random vectors whose components are not in $\mathbb{D}^{\infty}$, we recall the following general criterion for smoothness of densities established in \cite{FlN95}.

\begin{theorem}[{{\cite[Theorem 2.1]{FlN95} or \cite[Theorem 2.1.4]{Nua06}}}]\label{th2017-08-24-2}
Let $F = (F^1, \ldots, F^d)$ be a random vector whose components are in $\mathbb{D}^{1, 2}$. Let $A$ be an open subset of $\mathbb{R}^d$. Suppose that there exist $\mathscr{H}$-valued random variables $u_A^j, j = 1, \ldots, d$ and a $d \times d$ random matrix $\gamma_A = (\gamma_A^{i, j})$ such that
\begin{enumerate}
  \item [(i)] $u_A^j \in \mathbb{D}^{\infty}(\mathscr{H})$ for all $j = 1, \ldots, d$,
  \item [(ii)] $\gamma_A^{i, j} \in \mathbb{D}^{\infty}$ for all $i, j = 1, \ldots, d$, and $|\det \gamma_A|^{-1} \in L^p(\Omega)$ for all $p \geq 1$,
  \item [(iii)] $\langle DF^i, u_A^j \rangle_{\mathscr{H}} = \gamma_A^{i, j}$ on $\{F \in A\}$,  for all $i, j = 1, \ldots, d$.
\end{enumerate}
Then the random vector $F$ possesses an infinitely differentiable density on the open set $A$.
\end{theorem}
A random vector $F$ that satisfies the conditions  in Theorem \ref{th2017-08-24-2} is said to be {\em locally nondegenerate}.

Throughout the paper, the letters $C$, $c$, with or without index, will denote generic positive constants whose values may change from line to line, unless specified otherwise.

\section{Preliminaries}\label{section4.30}

In this section, we assume that $I$ and $J$ are as above \eqref{eq2018-03-01-600}. We will introduce two families of random variables to control the value of the supremum of the solution over a time interval and a space-time rectangle, respectively. For this purpose, we first introduce the Banach space $E_{p, \gamma}[a, b]$.

For an integer $p$, an arbitrary $\gamma \in \, ]\frac{1}{2p}, 1[$ and a continuous function $f$ defined on $[a, b]$, we define the H\"{o}lder seminorm
\begin{align}\label{eq11}
\|f\|_{p, \gamma} := \left(\int_{[a,b]^2}\frac{|f(x) - f(y)|^{2p}}{|x - y|^{1 + 2p\gamma}}dxdy\right)^{1/(2p)}.
\end{align}

Let $E_{p, \gamma}[a, b]$ denote the space of continuous functions vanishing at $a$ and having a finite $\|\cdot\|_{p, \gamma}$ norm. We omit $[a, b]$ if this interval is clear from the context. Each element of $E_{p, \gamma}$ turns out to be H\"{o}lder continuous. Indeed, we apply the Garsia, Rodemich and Rumsey lemma (see \cite[Lemma A.3.1]{Nua06}) to the real-valued function $f$ with $\Psi(x) = x^{2p}$, $p(x) = x^{(1 + 2p\gamma)/(2p)}$, $d = 1$ to find that there exists a constant $c$ such that for all $x$, $y \in [a, b]$,
\begin{align*}
|f(x) - f(y)| \leq c \, |x - y|^{\gamma - \frac{1}{2p}}\|f\|_{p, \gamma}.
\end{align*}
In fact, this is also a consequence of Sobolev embedding theorem; see \cite[Theorem 4.54]{DeD12}.
Moreover, as a fractional Sobolev space, $E_{p, \gamma}[a, b]$ is a separable Banach space; see \cite[Proposition 4.24]{DeD12}.

The following lemma gives an estimate on the rectangular increments of the solution of \eqref{eq2017-10-11-62}.
\begin{lemma} \label{lemma6}
There exists a constant $C_T$ such that for any $\theta \in\, ]0, \frac{1}{2}[$ and $(t, s, x, y) \in [0, T]^2 \times [0, 1]^2$,
\begin{align}
\mbox{E}[(u(t, x) + u(s, y) - u(t, y) - u(s, x))^2] &\leq C_T|t - s|^{\frac{1}{2}}\wedge |x - y| \nonumber\\
&\leq C_T|t - s|^{\frac{1}{2} - \theta}|x - y|^{2\theta}.  \label{eq2017-11-16-1}
\end{align}
\end{lemma}
\begin{proof}
The second inequality is trivial. To prove the first inequality, we apply the inequality $(a + b)^2 \leq 2(a^2 + b^2)$ to the left-hand side of  \eqref{eq2017-11-16-1}, first grouping the terms with the same space variable, then grouping the terms with the same time variable. Together with \eqref{eq2016-06-14-1}, the first case gives the bound $C_T|t - s|^{1/2}$, and the second case gives the bound $C_T|x - y|$. These two bounds establish the first inequality in \eqref{eq2017-11-16-1}.
\end{proof}

From now on, we fix $\theta \in \, ]0, \frac{1}{2}[$ and set
\begin{align}\label{eq2017-10-04-3}
\theta_1 = \frac{1}{2} - \theta, \quad \theta_2 = 2\theta.
\end{align}
By  the isometry and Lemma \ref{lemma6}, since $Du(t, x) = 1_{\{\cdot < t\}}G(t - \cdot, x, *)$ (here, the notations $\cdot$ and $*$ denote the time variable and space variable respectively), we have
\begin{align}\label{eq1001}
&\|D(u(t, x) + u(s, y) - u(t, y) - u(s, x))\|^2_{\mathscr{H}} \nonumber \\
&\quad = \mbox{E}[(u(t, x) + u(s, y) - u(t, y) - u(s, x))^2]  \leq C_T|t - s|^{\theta_1}|x - y|^{\theta_2},
\end{align}
for any $(t, s, x, y) \in [0, T]^2 \times [0, 1]^2$.

Using Lemma \ref{lemma6}, we establish the following property on the rectangular increment of sample path of the solution.
\begin{lemma}\label{lemma1001}
For any $0 < \xi < \theta_1/2$ and $0 < \eta < \theta_2/2$, there exists a random variable $C$ that is a.s. finite such that a.s., for all $(t, s, x, y) \in [0, T]^2 \times [0, 1]^2$,
\begin{align}\label{eq1002}
|u(t, x) + u(s, y) - u(t, y) - u(s, x)| \leq C|t - s|^{\xi}|x - y|^{\eta}.
\end{align}
\end{lemma}
\begin{remark}
A similar property is established in  \cite[Theorem 5.2]{HuL13}.
\end{remark}
\begin{proof}[Proof of Lemma \ref{lemma1001}]
Denote $\hat{u}(t, x) = u(t, x) - u(t, 0)$, $(t, x) \in [0, T] \times [0, 1]$.
We choose an integer $p$ and $\bar{\gamma}_2 \in \, ]0, 1[$ such that $\xi < \theta_1/2 - \frac{1}{2p}$ and $\eta + \frac{1}{2p} < \bar{\gamma}_2 <  \theta_2/2 - \frac{1}{2p}$.
Let $E_{p, \bar{\gamma}_2}[0, 1]$ be the space of continuous functions defined on $[0, 1]$ vanishing at $0$ and having a finite $\|\cdot\|_{p, \bar{\gamma}_2}$ norm.
Since a.s., for any $t \in [0, T]$, $x \mapsto \hat{u}(t, x)$ is almost $\frac{1}{2}$-H\"{o}lder continuous, we see that $\hat{u}(t, *)$ belongs to $E_{p, \bar{\gamma}_2}$. Moreover, by \eqref{eq2017-11-16-1}, for any $s$, $t \in [0, T]$,
\begin{align*}
\mbox{E}[\|\hat{u}(t, *) - \hat{u}(s, *)\|_{p, \bar{\gamma}_2}^{2p}] &= \int_{[0, 1]^2}\frac{\mbox{E}[|u(t, x) + u(s, y) - u(t, y) - u(s, x)|^{2p}]}{|x - y|^{1 + 2p\bar{\gamma}_2}}dxdy\\
&\leq C_T|t - s|^{\theta_1p}\int_{[0, 1]^2}\frac{|x - y|^{\theta_2p}}{|x - y|^{1 + 2p\bar{\gamma}_2}}dxdy\\
&\leq C_T|t - s|^{\theta_1p}.
\end{align*}
We apply the Kolmogorov continuity theorem (see \cite[Chapter I, Theorem 2.1]{ReY99}) to see that the process $\{\hat{u}(t, *): t \in [0, T]\}$ has a continuous version $\{\tilde{u}(t, *): t \in [0, T]\}$ with values in $E_{p, \bar{\gamma}_2}$, which is $\frac{\theta_1}{2} - \frac{1}{2p} - \epsilon$-H\"{o}lder continuous for small $\epsilon$ such that $\frac{\theta_1}{2} - \frac{1}{2p} - \epsilon > \xi$, namely, there exists a random variable $C$, finite almost surely, such that a.s. for any $s$, $t \in [0, T]$,
 \begin{align*}
 \|\tilde{u}(t, *) - \tilde{u}(s, *)\|_{p, \bar{\gamma}_2} \leq C|t - s|^{\frac{\theta_1}{2} - \frac{1}{2p} - \epsilon}.
 \end{align*}
 Hence we have for any $s$, $t \in [0, T]$,
 \begin{align*}
 \int_{[0, 1]^2}\frac{|\tilde{u}(t, x) - \tilde{u}(s, x) - \tilde{u}(t, y) + \tilde{u}(s, y)|^{2p}}{|x - y|^{1 + 2p\bar{\gamma}_2}}dxdy \leq C|t - s|^{(\frac{\theta_1}{2} - \frac{1}{2p} - \epsilon)2p}.
 \end{align*}
 We apply the Garsia, Rodemich and Rumsey lemma (see \cite[Lemma A.3.1]{Nua06}) to the real-valued function $x \mapsto \tilde{u}(t, x) - \tilde{u}(s, x)$ with $\Psi(x) = x^{2p}$, $p(x) = x^{(1 + 2p\bar{\gamma}_2)/(2p)}$, $d = 1$, to get that for any $(t, s, x, y) \in [0, T]^2 \times [0, 1]^2$,
 \begin{align} \label{eq110}
 |\tilde{u}(t, x) - \tilde{u}(s, x) - \tilde{u}(t, y) + \tilde{u}(s, y)| &\leq C|t - s|^{\frac{\theta_1}{2} - \frac{1}{2p} - \epsilon}|x - y|^{\bar{\gamma}_2 - \frac{1}{2p}} \nonumber\\
 &\leq C|t - s|^{\xi}|x - y|^{\eta}.
 \end{align}
 Letting $y = 0$ in (\ref{eq110}), we obtain
 \begin{align} \label{eq111}
 |\tilde{u}(t, x) - \tilde{u}(s, x)| \leq C|t - s|^{\xi}.
 \end{align}

Fix $(s, y)  \in [0, T] \times [0, 1]$.
Using the triangle inequality,
\begin{align*}
|\tilde{u}(t, x) - \tilde{u}(s, y)| \leq |\tilde{u}(t, x) - \tilde{u}(s, x)| + |\tilde{u}(s, x) - \tilde{u}(s, y)|,
 \end{align*}
which converges to $0$ as $(t, x) \rightarrow (s, y)$ by (\ref{eq111}) and the fact that $x \mapsto \tilde{u}(s, x)$ is continuous since $\tilde{u}(s, *) \in E_{p, \bar{\gamma}_2}$. Therefore, a.s., $(t, x) \mapsto \tilde{u}(t, x)$ is continuous. Together with the fact that for any $t \in [0, T]$, $\mbox{P}\{\hat{u}(t, *) = \tilde{u}(t, *)\} = 1$, we obtain that the processes $\{\hat{u}(t, x): (t, x) \in [0, T] \times [0, 1]\}$ and $\{\tilde{u}(t, x): (t, x) \in [0, T] \times [0, 1]\}$ are indistinguishable and hence (\ref{eq110}) implies (\ref{eq1002}).
\end{proof}

Choose $p_0 \in \mathbb{N}$ and $\gamma_0 \in \mathbb{R}$ such that
\begin{align}\label{eq2016-06-10-1}
p_0 - 2 > \gamma_0 > 4.
\end{align}
Let $\theta_1$ and $\theta_2$ be defined as in \eqref{eq2017-10-04-3}. We assume that $p_0$ is sufficiently large so that there exist $\gamma_1$, $\gamma_2$ such that
\begin{align}\label{eq2018-01-02-1}
\frac{1}{2p_0} < \gamma_1 < \frac{\theta_1}{2} - \frac{1}{2p_0}, \quad  \frac{1}{2p_0} < \gamma_2 < \frac{\theta_2}{2} - \frac{1}{2p_0},
\end{align}
and
\begin{align}\label{eq2018-01-02-2}
2\gamma_1 + \gamma_2 = \frac{\gamma_0 - 1}{2p_0};
\end{align}
see Figure \ref{figure4}.
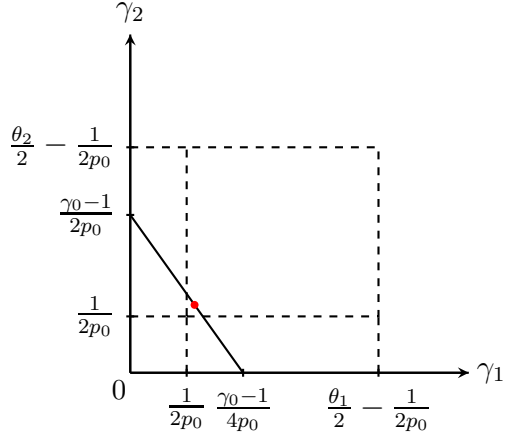
\begin{figure}
\centering {
\begin{tikzpicture}[>=stealth,scale=1.5,line width=0.8pt]


\pgfmathsetmacro{\ticker}{0.125} 

\coordinate (A) at (0,0);
\draw (-0.1,-0.15) node {{\small$0$}};
\coordinate (B) at (0,3);
\coordinate (C) at (3,3);
\coordinate (D) at (3,0);
\draw(A)--(B)--cycle;
\draw(A)--(D)--cycle;
\draw[->] (0, 0) -- (0, 3);
\draw[->] (0, 0) -- (3, 0);
\draw[dashed](0.5,0.5) rectangle (2.2, 2.0);
\draw (0,3.2) node { $\large{\gamma_2}$};

\draw[dashed] (0.5, 0) -- (0.5, 0.5);
\draw[dashed] (0, 0.5) -- (0.5, 0.5);
\draw[dashed] (2.2, 0) -- (2.2, 0.5);
\draw[dashed] (0, 2) -- (0.5, 2);

\draw (3.2,0) node { $\large{\gamma_1}$};

{\draw (0.5, 1pt) -- (0.5, -1pt) node[anchor=north] {$\frac{1}{2p_0}$};}
{\draw (1, 1pt) -- (1, -1pt) node[anchor=north] {$\frac{\gamma_0 - 1}{4p_0}$};}
{\draw (2.2, 1pt) -- (2.2, -1pt) node[anchor=north] {$\frac{\theta_1}{2} - \frac{1}{2p_0}$};}
{\draw (1pt,0.5) -- (-1pt,0.5) node[anchor=east] {$\frac{1}{2p_0}$};}
\draw (0, 1.4) -- (1, 0);

{\draw (1pt,1.4) -- (-1pt,1.4) node[anchor=east] {$\frac{\gamma_0 - 1}{2p_0}$};}
{\draw (1pt,2) -- (-1pt,2) node[anchor=east]  {$\frac{\theta_2}{2} - \frac{1}{2p_0}$};}
\fill[red](0.57,0.602) circle (1pt);

\end{tikzpicture}
}%
\caption{Illustration of \eqref{eq2018-01-02-1} and \eqref{eq2018-01-02-2}}\label{figure4}
\end{figure}

We now define two families of random variables $\{Y_{r}: r \in [s_0, s_0 + \delta_1]\}$ and $\{\bar{Y}_r: r \in [0, \Delta_{\bullet}]\}$, which will be used in Lemma  \ref{lemma11} below to  control respectively the value of the supremum of the solution over a time interval and over a space-time rectangle.
For $r \in [s_0, s_0 + \delta_1]$, define
\begin{align}\label{eq2017-09-28-1}
Y_{r} &:= \int_{[s_0, r]^2}\frac{(u(t, y_0) - u(s, y_0))^{2p_0}}{|t - s|^{\gamma_0/2}}dsdt.
\end{align}
Denote
\begin{align}\label{eq2018-01-03-1}
\delta: = \delta_1^{1/2} + \delta_2, \quad \Delta_{\bullet} : = \delta^2 \quad \mbox{and}\quad \Delta_{*}: = \delta \wedge (1 - y_0).
\end{align}
For $r \in [0, \Delta_{\bullet}]$, we define
\begin{align} \label{eq2018-01-03-2}
\bar{Y}_r := Y_0(r) + Y_1(r),
\end{align}
where
\begin{align}\label{eq2017-09-28-100}
Y_0(r) &:= \int_{[0, r]^2}\frac{(u(t, y_0) - u(s, y_0))^{2p_0}}{|t - s|^{\gamma_0/2}}dsdt,
\end{align}
and
\begin{align} \label{eq2018-10-10-1}
Y_1(r) &:= \int_{[0, r]^2}dtds \int_{[y_0, y_0 + \Delta_*]^2}dxdy\frac{(u(t, x) + u(s, y) - u(t, y) - u(s, x))^{2p_0}}{|t - s|^{1 + 2p_0\gamma_1}|x - y|^{1 + 2p_0\gamma_2}}.
\end{align}

\begin{lemma}\label{lemma2018-10-09-1}
\begin{itemize}
  \item [(a)] For any $p \geq 1$, there exists a constant $c_p$, not depending on $(s_0, y_0) \in [0, T] \times [0, 1]$, such that for all $r \in [s_0, s_0 + \delta_1]$,
      \begin{align}\label{eq2018-01-03-6}
\mbox{E}[|Y_{r}|^p] &\leq c_p \, (r - s_0)^{2p}\delta_1^{(p_0 - \gamma_0)p/2}.
\end{align}
  \item [(b)] For any $p \geq 1$, there exists a constant $c_p$, not depending on $y_0 \in [0, 1]$,  such that for any $r \in [0, \Delta_{\bullet}]$,
\begin{align} \label{eq2018-01-03-8}
\mbox{E}[|\bar{Y}_r|^p] &\leq c_p \, r^{2p}\delta^{p(p_0 - \gamma_0)}.
\end{align}
\end{itemize}
\end{lemma}
\begin{proof}
We first prove (a). By H\"{o}lder's inequality and \eqref{eq2016-06-14-1}, for any $p \geq 1$,
\begin{align*}
\mbox{E}[|Y_{r}|^p] &\leq (r - s_0)^{2(p - 1)}\int_{[s_0, r]^2}\frac{\mbox{E}[|u(t, y_0) - u(s, y_0)|^{2p_0p}]}{|t - s|^{\gamma_0p/2}}dsdt \nonumber \\
& \leq c_p \, (r - s_0)^{2(p - 1)}\int_{[s_0, r]^2}\frac{|t - s|^{p_0p/2}}{|t - s|^{\gamma_0p/2}}dsdt \leq c_p \, (r - s_0)^{2p}\delta_1^{(p_0 - \gamma_0)p/2},
\end{align*}
where the constant $c_p$ does not depend on $(s_0, y_0) \in [0, T] \times [0, 1]$.

To prove (b), it suffices to estimate the moments of $Y_1(r)$ since the estimate of the moments of $Y_0(r)$ follows along the same lines as that of $Y_r$. By H\"{o}lder's inequality and \eqref{eq1001}, for any $p \geq 1$, there exists a constant $c_p$, not depending on $y_0 \in [0, 1]$, such that for any $r \in [0, \Delta_{\bullet}]$,
\begin{align} \label{eq2018-01-03-7}
\mbox{E}[|Y_1(r)|^p] &\leq (r\Delta_*)^{2(p - 1)}\int_{[0, r]^2}dtds \int_{[y_0, y_0 + \Delta_*]^2}dxdy \nonumber \\
& \qquad \qquad \qquad \qquad \qquad \qquad\times \frac{\mbox{E}[|u(t, x) + u(s, y) - u(t, y) - u(s, x)|^{2p_0p}]}{|t - s|^{p(1 + 2p_0\gamma_1)}|x - y|^{p(1 + 2p_0\gamma_2)}}\nonumber \\
& \leq c_p \, (r\Delta_*)^{2(p - 1)}\int_{[0, r]^2}dtds \int_{[y_0, y_0 + \Delta_*]^2}dxdy \frac{|t - s|^{p_0p\theta_1}|x - y|^{p_0p\theta_2}}{|t - s|^{p(1 + 2p_0\gamma_1)}|x - y|^{p(1 + 2p_0\gamma_2)}}\nonumber \\
& \leq c_p \, (r\Delta_*)^{2p} \Delta_{\bullet}^{p(p_0\theta_1 - (1 + 2p_0\gamma_1))}\Delta_*^{p(p_0\theta_2 - (1 + 2p_0\gamma_2))}\nonumber \\
& \leq c_p \, r^{2p} \delta^{p(2p_0\theta_1 - 2(1 + 2p_0\gamma_1))}\delta^{p(p_0\theta_2 - (1 + 2p_0\gamma_2) + 2)}\nonumber \\
& = c_p \, r^{2p}\delta^{p(p_0(2\theta_1 + \theta_2) - 2p_0 (2\gamma_1 + \gamma_2) - 1)}  = c_p \, r^{2p}\delta^{p(p_0 - \gamma_0)},
\end{align}
where in the last inequality we use \eqref{eq2018-01-03-1}, and in the second equality we use \eqref{eq2018-01-02-2} and the fact that $2\theta_1 + \theta_2 = 1$ from the definition of $\theta_1$, $\theta_2$ in \eqref{eq2017-10-04-3}.
\end{proof}

The random variable $Y_1(r)$ defined in \eqref{eq2018-10-10-1} has another representation in terms of the H\"{o}lder seminorm $\|\cdot\|_{p_0, \gamma_2}$. Indeed, we will write, for $(t, x) \in [0, T] \times [0, 1]$,
\begin{align}
u(t, x) = \check{u}(t, x) + u(t, y_0),
\end{align}
where
\begin{align} \label{eq2018-01-16-1}
\check{u}(t, x) = u(t, x) - u(t, y_0).
\end{align}
Since for any $\epsilon > 0$, a.s., for any fixed $t$, the function $x \mapsto \check{u}(t, x)$ is $\frac{1}{2}-\epsilon$-H\"{o}lder continuous, it follows that $\check{u}(t, *)$ belongs to the Banach space $E_{p_0, \gamma_2}[y_0, y_0 + \Delta_*]$ (the space of continuous functions defined on $[y_0, y_0 + \Delta_*]$ vanishing at $y_0$ and having a finite $\|\cdot\|_{p_0, \gamma_2}$ norm) with $p_0, \, \gamma_2$ as defined in \eqref{eq2016-06-10-1} and \eqref{eq2018-01-02-1}.
We can write, for $r \in [0, \Delta_{\bullet}]$,
\begin{align}\label{eq2017-10-05-100}
Y_1(r) &= \int_{[0, r]^2}dtds \int_{[y_0, y_0 + \Delta_*]^2}dxdy \, \frac{(u(t, x) + u(s, y) - u(t, y) - u(s, x))^{2p_0}}{|t - s|^{1 + 2p_0\gamma_1}|x - y|^{1 + 2p_0\gamma_2}} \nonumber \\
&=  \int_{[0, r]^2}\frac{\|\check{u}(t, *) - \check{u}(s, *)\|_{p_0, \gamma_2}^{2p_0}}{|t - s|^{1 + 2p_0\gamma_1}}dtds.
\end{align}

Moreover, choose $\xi, \, \eta$ as in Lemma \ref{lemma1001} such that $\eta > \gamma_2 + 1/(2p_0)$, which is possible by \eqref{eq2018-01-02-1}. Then, by \eqref{eq1002},
\begin{align}\label{eq2018-03-02-1}
\|\check{u}(t, *) - \check{u}(s, *)\|_{p_0, \gamma_2}^{2p_0} & = \int_{[y_0, y_0 + \Delta_*]^2}\frac{(u(t, x) + u(s, y) - u(t, y) - u(s, x))^{2p_0}}{|x - y|^{1 + 2p_0\gamma_2}}dxdy \nonumber \\
& \leq C|t - s|^{2p_0\xi}\int_{[y_0, y_0 + \Delta_*]^2}|x - y|^{2p_0\eta - 1 - 2p_0\gamma_2}dxdy \nonumber \\
& \leq C|t - s|^{2p_0\xi}
\end{align}
since $2p_0\eta - 1 - 2p_0\gamma_2 > 0$, which shows that a.s., $t \mapsto \check{u}(t, *)$ is continuous in $E_{p_0, \gamma_2}[y_0, y_0 + \Delta_*]$.

The following result shows that the two families of random variables $\{Y_{r}: r \in [s_0, s_0 + \delta_1]\}$ and $\{\bar{Y}_r: r \in [0, \Delta_{\bullet}]\}$ can control respectively the value of the supremum of the solution over a time interval and over a space-time rectangle.
\begin{lemma} \label{lemma11}
\begin{itemize}
  \item [(a)] There exists a finite positive constant $c$, not depending on $(s_0, y_0) \in [0, T] \times [0, 1]$, such that for any $a > 0$, for all $\delta_1 > 0$ and for all $r \in [s_0, s_0 + \delta_1]$,
\begin{align} \label{eq2017-10-12-5}
Y_{r} \leq R:= c\, a^{2p_0}\delta_1^{-(\gamma_0 - 4)/2} \quad \Rightarrow \quad \sup_{
 t \in [s_0, r]
}|\bar{u}(t, y_0)| \leq a.
\end{align}
  \item [(b)]There exists a finite positive constant $c$, not depending on $y_0 \in [0, 1]$, such that for any $\bar{a} > 0$, $\delta_1 > 0$, $\delta_2 > 0$ and for all $r \in [0, \Delta_{\bullet}]$,
\begin{align} \label{eq2017-10-12-500}
\bar{Y}_{r} \leq \bar{R}:= c\, \bar{a}^{2p_0}\delta^{4 -\gamma_0} \quad \Rightarrow \quad \sup_{
 (t,x) \in [0, r] \times [y_0, y_0 + \delta_2]
 }|u(t, x)| \leq \bar{a}.
\end{align}
\end{itemize}
\end{lemma}
\begin{proof}
We first prove (a). We apply the Garsia, Rodemich, and Rumsey lemma (see \cite[Proposition A.1]{DKN07}) with
\begin{align*}
& S := [s_0, r], \,\, \rho(t, s) := |t - s|^{1/2}, \,\, \mu(dt) := dt, \\
&\Psi(x) := x^{2p_0}, \,\, p(x) := x^{\frac{\gamma_0}{2p_0}} \quad \mbox{and}  \quad f := u(\cdot, y_0).
\end{align*}
By \cite[(A.2)]{DKN07}, we deduce that for all $t$, $s \in [s_0, r]$,
\begin{align*}
|u(t, y_0) - u(s, y_0)| &\leq 10  \int_0^{2\rho(t, s)}\frac{Y_{r}^{\frac{1}{2p_0}}}{[\mu(B_{\rho}(s, u/4))]^{1/p_0}}u^{\frac{\gamma_0}{2p_0} - 1}du   \\
&\leq c_1 \, Y_{r}^{\frac{1}{2p_0}}\int_0^{2\rho(t, s)}u^{-\frac{2}{p_0}}u^{\frac{\gamma_0}{2p_0} - 1}du   = c_2 \,|t - s|^{\frac{\gamma_0 - 4}{4p_0}}Y_{r}^{\frac{1}{2p_0}}\leq  c_2 \,\delta_1^{\frac{\gamma_0 - 4}{4p_0}}Y_{r}^{\frac{1}{2p_0}},
\end{align*}
where we have used \eqref{eq2016-06-10-1}; the constants $c_1, \, c_2$ do not depend on $r$, nor on $(s_0, y_0) \in [0, T] \times [0, 1]$.
Assuming $Y_{r} \leq R$, letting $s = s_0$ in the above inequality and choosing a suitable constant in the definition of $R$, we obtain that
\begin{equation*}
\sup_{ t \in [s_0, r]
}|\bar{u}(t, y_0)| \leq \delta_1^{(\gamma_0 - 4)/(4p_0)}(a^{2p_0}\delta_1^{-(\gamma_0 - 4)/2})^{1/(2p_0)} = a.
\end{equation*}

We now prove (b). Assuming $Y_0(r) \leq \bar{R}$, similar to the proof of (a), by the Garsia, Rodemich and Rumsey lemma (see \cite[Proposition A.1]{DKN07}), we deduce that for all $t$, $s \in [0, r]$,
\begin{align}\label{eq2018-01-03-3}
|u(t, y_0) - u(s, y_0)| &\leq  c' \,|t - s|^{\frac{\gamma_0 - 4}{4p_0}}Y_0(r)^{\frac{1}{2p_0}}\leq  c_1 \,\Delta_{\bullet}^{\frac{\gamma_0 - 4}{4p_0}}Y_0(r)^{\frac{1}{2p_0}} = c_1 \,\delta^{\frac{\gamma_0 - 4}{2p_0}}Y_0(r)^{\frac{1}{2p_0}},
\end{align}
where the constant $c_1$ does not depend on $r$, nor on $y_0 \in [0, 1]$. Letting $s = 0$ in \eqref{eq2018-01-03-3}, we obtain
\begin{align}\label{eq2018-01-03-4}
\sup_{t \in [0, r]} |u(t, y_0)| \leq c_1 \,\delta^{\frac{\gamma_0 - 4}{2p_0}}Y_0(r)^{\frac{1}{2p_0}} \leq c_1 \,\delta^{\frac{\gamma_0 - 4}{2p_0}}\bar{R}^{\frac{1}{2p_0}}.
\end{align}
Hence we can choose a suitable constant $c$ in the definition of $\bar{R}$ in \eqref{eq2017-10-12-500} so that
\begin{align}\label{eq2018-01-03-5}
\sup_{t \in [0, r]} |u(t, y_0)| \leq \bar{a}/2.
\end{align}

Assuming $Y_1(r) \leq \bar{R}$, from the expression of $Y_1(r)$ in \eqref{eq2017-10-05-100}, we first apply the Garsia, Rodemich, and Rumsey lemma (see \cite[Lemma A.3.1]{Nua06}) to the $E_{p_0, \gamma_2}[y_0, y_0 + \Delta_*]$-valued function $s \mapsto \check{u}(s, *)$ with $\Psi(x) = x^{2p_0}$, $p(x) = x^{(1 + 2p_0\gamma_1)/(2p_0)}$ to deduce that there exists a constant $c_2$ such that for all $t$, $s \in [0, r]$,
\begin{align}\label{eq2018-03-01-8}
\|\check{u}(t, *) - \check{u}(s, *)\|_{p_0, \gamma_2} &\leq  c' \, Y_1(r)^{\frac{1}{2p_0}} \int_0^{2|t - s|}x^{-\frac{1}{p_0}}x^{\frac{1 + 2p_0\gamma_1}{2p_0} - 1}dx \nonumber \\
& = c_2 \, Y_1(r)^{\frac{1}{2p_0}}|t - s|^{\frac{2p_0\gamma_1 -1}{2p_0}}\nonumber\\
& \leq c_2 \, Y_1(r)^{\frac{1}{2p_0}}\Delta_{\bullet}^{\frac{2p_0\gamma_1 -1}{2p_0}} = c_2 \, Y_1(r)^{\frac{1}{2p_0}}\delta^{\frac{2(2p_0\gamma_1 -1)}{2p_0}}.
\end{align}
Letting $s = 0$, we obtain for all $t \in [0, r]$,
\begin{align*}
\|\check{u}(t, *)\|_{p_0, \gamma_2}^{2p_0} &\leq  c_2 \, Y_1(r)\delta^{2(2p_0\gamma_1 -1)}.
\end{align*}
Applying the same lemma to the real-valued function $x \mapsto \check{u}(t, x)$ ($t$ is now fixed) with $\Psi(x) = x^{2p_0}$, $p(x) = x^{(1 + 2p_0\gamma_2)/(2p_0)}$, we obtain
\begin{align*}
|\check{u}(t, x) - \check{u}(t, y)| \leq c_3 \, Y_1(r)^{\frac{1}{2p_0}}\delta^{\frac{2(2p_0\gamma_1 -1)}{2p_0}} |x - y|^{\frac{2p_0\gamma_2 -1}{2p_0}},
\end{align*}
for all $x$, $y \in [y_0, y_0 + \Delta_*]$. Letting $y = y_0$ we obtain that for all $(t, x) \in [0, r] \times [y_0, y_0 + \Delta_*]$,
\begin{align*}
|u(t, x) - u(t, y_0)| & \leq c_3 \, Y_1(r)^{\frac{1}{2p_0}}\delta^{\frac{2(2p_0\gamma_1 -1)}{2p_0}} \Delta_*^{\frac{2p_0\gamma_2 -1}{2p_0}}\\
&\leq  c_3 \, Y_1(r)^{\frac{1}{2p_0}}\delta^{\frac{2(2p_0\gamma_1 -1)}{2p_0}} \delta^{\frac{2p_0\gamma_2 -1}{2p_0}} = c_3 \, Y_1(r)^{\frac{1}{2p_0}} \delta^{\frac{\gamma_0 - 4}{2p_0}},
\end{align*}
where in the second inequality we use \eqref{eq2018-01-03-1}, and the equality is due to \eqref{eq2018-01-02-2}.
In particular, this implies that
\begin{align}\label{eq2018-01-02-300}
\sup_{(t, x) \in [0, r] \times [y_0, y_0 + \delta_2]}|u(t, x) - u(t, y_0)| &\leq  c_3 \, Y_1(r)^{\frac{1}{2p_0}} \delta^{\frac{\gamma_0 - 4}{2p_0}}.
\end{align}
We can choose the constant $c$ in the definition of $\bar{R}$ in \eqref{eq2017-10-12-500} small so that \eqref{eq2018-01-03-5} holds and
\begin{align}\label{eq2018-01-02-30011}
\sup_{(t, x) \in [0, r] \times [y_0, y_0 + \delta_2]}|u(t, x) - u(t, y_0)| &\leq  \bar{a}/2.
\end{align}
Hence, by \eqref{eq2018-01-03-5}, \eqref{eq2018-01-02-30011} and the triangle inequality, we obtain \eqref{eq2017-10-12-500}.
\end{proof}

We conclude this section by introducing a result on the uniqueness of the solution to the heat equation with boundary conditions, which will be used when we check the condition (iii) of Theorem \ref{th2017-08-24-2}.

Let $f: [0, \infty[ \mapsto \mathbb{R}$ be a differentiable function with continuous derivative satisfying $f(0) = 0$. Let $g \in C^{\infty}([0, 1])$ satisfy the same boundary conditions as the Green kernel.  We define
\begin{align*}
& A(t, x) = \int_0^t\int_0^1  G(t - r, x, v)\left(\frac{\partial}{\partial r} - \frac{\partial}{\partial v^2}\right)(f(r)g(v))dvdr, \quad t > 0, \, x \in [0, 1],\\
&  A(0, x) = 0, \quad x \in [0, 1].
\end{align*}

\begin{lemma} \label{lemma2017-11-27-1}
The function $A$ is well-defined and we have $A(t, x) = f(t)g(x)$ for all $(t, x) \in [0, \infty[ \, \times [0, 1]$.
\end{lemma}
\begin{proof}
It is clear that the function $A$ is well-defined since both the Green kernel and the function $(r, v) \mapsto \left(\frac{\partial}{\partial r} - \frac{\partial}{\partial v^2}\right)(f(r)g(v))$ belong to $L^2([0, T] \times [0, 1])$. From the definition of the function $A$, we see that $A$ solves the inhomogeneous heat equation, that is, $A$ satisfies
\begin{align} \label{eq2017-11-27-2}
\left(\frac{\partial}{\partial t} - \frac{\partial}{\partial x^2}\right)A(t, x) = \left(\frac{\partial}{\partial t} - \frac{\partial}{\partial x^2}\right)(f(t)g(x)),
\end{align}
the same boundary conditions as the Green kernel and vanishing initial condition. On the other hand, the function $f(\cdot)g(*)$ also satisfies \eqref{eq2017-11-27-2} with $A(t, x)$ replaced by $f(t)g(x)$ and the same boundary and initial conditions. By the uniqueness of the solution to heat equation on bounded domains (see \cite[Theorem 5, p.57]{Eva98}), we have $A = f(\cdot)g(*)$.
\end{proof}

\section{Malliavin derivatives of $F_2$ and $M_0$}

In this section, we  recall some results on the suprema $F_2$ and $M_0$ in \eqref{eq2017-12-19-2}--\eqref{eq2018-01-02-4}, in order to apply Theorem \ref{th2017-08-24-2} and to prove Theorems \ref{the2017-11-25-1}, \ref{theorem100} and \ref{theorem2018-01-03-1}.

First, we state the 0-1 law for the germ $\sigma$-algebra generated by the Brownian sheet that appears in equation \eqref{eq2017-10-11-62}.
For $t \geq 0$, define the filtration,
\begin{align}\label{eq2018-07-31-1}
\mathscr{F}_t := \sigma\{W(s, x): s \leq t, 0 \leq x \leq 1\} \vee \mathcal{N} \quad \mbox{and} \quad \mathscr{F}_t^+ := \bigcap_{s > t}\mathscr{F}_s,
 \end{align}
 where $\mathcal{N}$ is the $\sigma$-field generated by $\mbox{P}$-null sets.

\begin{lemma}\label{prop3.1}
For any set $B \in \mathscr{F}_0^+, \mbox{P}(B) \in \{0, 1\}.$
\end{lemma}
\noindent{}This lemma is a consequence of \cite[Proposition 3.2]{Wal861}.

\begin{lemma}\label{lemma100}
Fix $(s_0, y_0) \in I \times J$.
\begin{itemize}
  \item [(a)] With probability one, each sample path of the process $\{\bar{u}(t, y_0): t \in [s_0, s_0 + \delta_1] \}$ achieves its supremum at a unique point in $[s_0, s_0 + \delta_1]$, denoted by $S$. If $y_0 \in \, ]0, 1[$, then $F_2 > 0$ a.s.
       \item [(b)]With probability one, $M_0 > 0$ and each sample path of the process $\{u(t, x): (t, x) \in [0, \delta_1] \times [y_0, y_0 + \delta_2]\}$ achieves its supremum at a unique point in $\, ]0, \delta_1] \times [y_0, y_0 + \delta_2]$, denoted by $(\bar{S}, \bar{X})$.
  \end{itemize}

\end{lemma}
\begin{proof}
The first statement of (a) follows from \cite[Lemma 2.6]{KiP}, since for $t, s \in [s_0, s_0 + \delta_1]$ with $t \neq s$,
\begin{align*}
\mbox{E}[|\bar{u}(t, y_0) - \bar{u}(s, y_0)|^2] = \mbox{E}[|u(t, y_0) - u(s, y_0)|^2] \neq 0,
\end{align*}
by \cite[Lemma 4.2]{DKN07}. For the second statement of (a),  if $s_0 = 0$, it is clear that $F_2 > 0$ a.s. by using the 0-1 law in Lemma \ref{prop3.1}; see also the proof for $M_0 >0$ a.s. below. If $s_0 > 0$, we denote $\{\tilde{u}(t, x): (t, x) \in [s_0, s_0 + \delta_1] \times [y_0, y_0 + \delta_2]\}$ the solution to \eqref{eq2017-10-11-62} on the whole real line $\mathbb{R}$. By \cite[p.23, (3.9)]{Kho14}, we see that
\begin{align*}
\sup_{t \in [s_0, s_0 + \delta_1]}\tilde{u}(t, y_0) - \tilde{u}(s_0, y_0) > 0 \quad \mbox{a.s.}
\end{align*}
Since the processes $\{u(t, x): (t, x) \in [s_0, s_0 + \delta_1] \times [y_0, y_0 + \delta_2]\}$ and $\{\tilde{u}(t, x): (t, x) \in [s_0, s_0 + \delta_1] \times [y_0, y_0 + \delta_2]\}$ are mutually absolute continuous by \cite[Corollary 4]{MuT02}, we conclude that $F_2 > 0$ a.s.

We turn to proving statement (b). Fix $x \in [y_0, y_0 + \delta_2]$. It is clear that
\begin{align} \label{eq2017-11-30-2}
\{M_0 > 0\} = \Big\{\sup_{(t, x) \in [0, \delta_1] \times [y_0, y_0 + \delta_2]}u(t , x) > 0\Big\} \supset \limsup_{t_n \downarrow 0}\, \{u(t_n, x) > 0\}.
 \end{align}
 On the other hand, we know that
 \begin{align} \label{eq2017-11-30-3}
 \limsup_{t_n \downarrow 0}\, \{u(t_n, x) > 0\} \in \mathscr{F}_0^+
  \end{align}
  and
  \begin{align}\label{eq2017-11-30-4}
  \mbox{P}\Big\{\limsup_{t_n \downarrow 0}\, \{u(t_n, x) > 0\}\Big\} \geq \limsup\limits_{t_n \downarrow 0}\mbox{P}\{u(t_n, x) > 0\} = \frac{1}{2},
   \end{align}
since for every $n$, $u(t_n, x)$ is a centered Gaussian random variable and $\mbox{P}\{u(t_n, x) > 0\} = \frac{1}{2}$.
Hence by Lemma \ref{prop3.1}, we obtain that
\begin{align}\label{eq2017-11-30-5}
\mbox{P}\Big\{\limsup_{t_n \downarrow 0}\, \{u(t_n, x) > 0\}\Big\} = 1,
\end{align}
 which establishes that $M_0 > 0$ a.s. Furthermore, for any $(t, x)$, $(s, y) \in \, ]0, \delta_1] \times [y_0, y_0 + \delta_2]$ with  $(t, x) \neq (s, y)$, by \cite[Lemma 4.2]{DKN07},
\begin{align*}
\mbox{E}[|u(t, x) - u(s, y)|^2] \neq 0,
\end{align*}
which yields the conclusion of statement (b) by \cite[ Lemma 2.6 ]{KiP}.
\end{proof}

\begin{lemma} \label{lemma9}
The random variables $M_0$ and $F_2$ belong to $\mathbb{D}^{1, 2}$ and
\begin{align}
DM_0 & = 1_{\{\cdot < \bar{S}\}}G(\bar{S} - \cdot, \bar{X}, *),\label{eq2017-09-27-20011} \\
DF_2 & = 1_{\{\cdot < S\}}G(S - \cdot, y_0, *) - 1_{\{\cdot < s_0\}}G(s_0 - \cdot, y_0, *), \label{eq2017-09-27-2}
\end{align}
where the random variables $\bar{S}$, $\bar{X}$ and $S$ are defined in Lemma \ref{lemma100}.
\end{lemma}

\begin{remark}
The function $(t, x) \mapsto 1_{\{\cdot < t\}}G(t - \cdot, x, *)$ from $[0, T] \times [0, 1]$ into $\mathscr{H}$ is continuous by \eqref{eq2017-09-28-2}. Therefore, $1_{\{\cdot < \bar{S}\}}G(\bar{S} - \cdot, \bar{X}, *)$ is the random element of $\mathscr{H}$ obtained by composition of the random vector $\omega \mapsto (\bar{S}(\omega), \bar{X}(\omega))$ and this continuous function.
\end{remark}
\begin{proof}[Proof of Lemma \ref{lemma9}]
It is similar to the proof for the Brownian sheet; see \cite[Lemma 2.1.9]{Nua06}. We present the main ingredients and refer to \cite{Pu18} for full details. We only prove \eqref{eq2017-09-27-20011} since the proof of \eqref{eq2017-09-27-2} is similar.
Let $\{(t_k, x_k)\}_{k = 1}^{\infty}$ be a dense subset of $[0, T] \times [0, 1]$. Define
\begin{align*}
M_n := \max\{u(t_1, x_1), \ldots, u(t_n, x_n)\}.
\end{align*}
Then $M_n$ converges to $M$ almost surely as $n \rightarrow \infty$. Using the local property of the operator $D$ (see Proposition 1.3.16 in \cite{Nua06}), we see that
\begin{align}\label{eq1006}
DM_n &=1_{\{\cdot < S_n\}}G(S_n - \cdot, X_n, *),
\end{align}
where $(S_n, X_n)$ is the unique point such that $M_n = u(S_n, X_n)$.
By checking the conditions of \cite[Lemma 1.2.3]{Nua06}, we see that $M_0$ belongs to $\mathbb{D}^{1, 2}$ and $DM_n$ converges to $DM_0$ in the weak topology of $L^2(\Omega, \mathscr{H})$. Using the fact that the function $(t, x) \mapsto Du(t, x) = 1_{\{\cdot < t\}} G(t - \cdot, x, *)$ from $[0, T] \times [0, 1]$ into $\mathscr{H}$ is continuous, because for any $(t, x), (s, y) \in [0, T] \times [0, 1]$, by \eqref{eq2016-06-14-1},
\begin{align}\label{eq2017-09-28-2}
\|D(u(t, x) - u(s, y))\|_{\mathscr{H}}^2 &= \|1_{\{\cdot < t\}} G(t - \cdot, x, *) - 1_{\{\cdot < s\}} G(s - \cdot, y, *)\|_{\mathscr{H}}^2\nonumber \\
&= \mbox{E}[|u(t, x) - u(s, y)|^2]\\
&\leq C_T(|t - s|^{1/2} + |x - y|), \nonumber
\end{align}
and since $(S_n, X_n)$ converges to $(\bar{S}, \bar{X})$ a.s., we conclude that
$
 DM_0 = 1_{\{\cdot < \bar{S}\}}G(\bar{S} - \cdot, \bar{X}, *).
 $
\end{proof}

\section{Smoothness of the densities} \label{section4.3}
In this section, we suppose that $I$ and $J$ are as above \eqref{eq2018-03-01-600} and we are going to introduce the random variables needed for Theorem \ref{th2017-08-24-2} and prove they satisfy the conditions therein. We start by establishing the smoothness of the random variables $\{Y_r: r \in [s_0, s_0 + \delta_1]\}$ and $\{\bar{Y}_r: r \in [0, \Delta_{\bullet}]\}$ defined in \eqref{eq2017-09-28-1} and \eqref{eq2018-01-03-2} respectively.

For simplicity of notation, we denote, for $(t, s, x, y) \in [0, T]^2 \times [0, 1]^2$,
\begin{align*}
u(1_{]s, t]\times]y, x]})&:= u(t, x) + u(s, y) - u(t, y) - u(s, x), \\
Du(t, x; s, y)&:= D(u(t, x) + u(s, y) - u(t, y) - u(s, x)).
\end{align*}

\begin{lemma} \label{lemma13}
\begin{itemize}
  \item [(a)]
  For any  $r \in [s_0, s_0 + \delta_1]$, $Y_{r}$ belongs to $\mathbb{D}^{\infty}$ and for any integer $l$,
\begin{align}\label{eq2017-09-28-3}
D^lY_{r}
&= \int_{[s_0, r]^2}dtds \, \frac{2p_0(2p_0 - 1) \cdots (2p_0 - l + 1)}{|t - s|^{\gamma_0/2}}\nonumber \\
& \quad \quad \quad \quad \quad \times (u(t, y_0) - u(s, y_0))^{2p_0 - l}(D(u(t, y_0) - u(s, y_0)))^{\otimes l}.
\end{align}
\item [(b)]  For any  $r \in [0, \Delta_{\bullet}]$, $\bar{Y}_{r}$ belongs to $\mathbb{D}^{\infty}$ and for any integer $l$,
\begin{align}\label{eq2017-09-28-300}
D^l\bar{Y}_{r}
&= \int_{[0, r]^2}dtds \,  \frac{2p_0(2p_0 - 1) \cdots (2p_0 - l + 1)}{|t - s|^{\gamma_0/2}}\nonumber \\
& \quad \quad \quad \quad \quad  \times (u(t, y_0) - u(s, y_0))^{2p_0 - l}(D(u(t, y_0) - u(s, y_0)))^{\otimes l}\nonumber \\
& \quad + \int_{[0, r]^2}dtds \int_{[y_0, y_0 + \Delta_*]^2}dxdy \, \frac{2p_0(2p_0 - 1)\cdots(2p_0 - l + 1)}{|t - s|^{1 + 2p_0\gamma_1}|x - y|^{1 + 2p_0\gamma_2}} \nonumber \\
& \quad \quad \quad \quad \quad \quad \quad \quad \quad \quad  \times u(1_{]s, t]\times]y, x]})^{2p_0 - l}(Du(t, x; s, y))^{\otimes l}.
\end{align}
\end{itemize}
\end{lemma}
\begin{proof}
Lemma \ref{lemma13} follows by permuting the Malliavin derivative and Lebesgue integral; see \cite[Proposition 3.4.3]{Nua18} for such a property. We refer to \cite{Pu18} for a detailed proof.
\end{proof}

Moreover, we have the following estimates on moments of the Malliavin derivatives of the random variables $\{Y_r, \, r \in [s_0, s_0 + \delta_1]\}$ and $\{\bar{Y}_r, \, r \in [0, \Delta_{\bullet}]\}$.

\begin{lemma} \label{lemma2017-10-2-1}
\begin{itemize}
  \item [(a)]For any $p \geq 1$,  there exists a constant $c_p$,  not depending on $(s_0, y_0) \in [0, T] \times [0, 1]$, such that for all $\delta_1 > 0$ and for all $r \in [s_0, s_0 + \delta_1]$,
\begin{align}\label{eq2017-10}
\mbox{E}[\|DY_r\|^p_{\mathscr{H}}]\leq c_p(r - s_0)^{2p}\delta_1^{(p_0 - \gamma_0)p/2}.
\end{align}
  \item [(b)]For any $p \geq 1$, there exists a constant $c_p$, not depending on $y_0 \in [0, 1]$,  such that  for all $r \in [0, \Delta_{\bullet}]$,
\begin{align}\label{eq2018-01-03-10}
\mbox{E}[\|D\bar{Y}_r\|_{\mathscr{H}}^p] &\leq c_p \, r^{2p}\delta^{(p_0 - \gamma_0)p}.
\end{align}
\item [(c)] For any integer $i$ and $p \geq 1$,
\begin{align} \label{eq2017-11-26-9}
 \sup_{r \in [s_0, s_0 + \delta_1]}\mbox{E}\left[\|D^iY_{r}\|_{\mathscr{H}^{\otimes i}}^p\right] < \infty \quad \mbox{and}\quad  \sup_{r \in [0, \Delta_{\bullet}]}\mbox{E}\left[\|D^i\bar{Y}_{r}\|_{\mathscr{H}^{\otimes i}}^p\right] < \infty.
\end{align}
\end{itemize}
\end{lemma}
\begin{proof}
We first prove \eqref{eq2017-10}. By Lemma \ref{lemma13}(a),
\begin{align}\label{eq2017-11}
DY_r = 2p_0\int_{[s_0, r]^2}dsdt\, \frac{(u(t, y_0) - u(s, y_0))^{2p_0 - 1}}{|t - s|^{\gamma_0/2}}D(u(t, y_0) - u(s, y_0)),
\end{align}
and for any $p \geq 1$, by H\"{o}lder's inequality,
\begin{align}\label{eq2017-12}
\mbox{E}[\|DY_r\|^p_{\mathscr{H}}] &\leq  c_p (r - s_0)^{2(p - 1)}\int_{[s_0, r]^2}dsdt\, \frac{\mbox{E}\left[|u(t, y_0) - u(s, y_0)|^{(2p_0 - 1)p}\right]}{|t - s|^{\gamma_0p/2}}\nonumber \\
& \quad \quad \quad \quad\quad \quad\quad \quad\quad \quad \quad \quad\quad \quad\quad \quad\times \|D(u(t, y_0) - u(s, y_0))\|^p_{\mathscr{H}}\nonumber \\
& \leq c_p (r - s_0)^{2(p - 1)}\int_{[s_0, r]^2}dsdt \, |t - s|^{p(p_0 - \gamma_0)/2}\leq c_p (r - s_0)^{2p}\delta_1^{(p_0 - \gamma_0)p/2},
\end{align}
where in the second inequality we use \eqref{eq2017-09-28-2} and \eqref{eq2016-06-14-1}.

To prove (b), it suffices to estimate the moments of $DY_1(r)$ since the estimate for the moments of $DY_0(r)$ is similar to the proof of (a). Indeed, from \eqref{eq2017-09-28-300}, by H\"{o}lder's inequality, for any $p \geq 1$,
\begin{align}\label{eq2018-01-25-2}
\mbox{E}[\|DY_1(r)\|^p_{\mathscr{H}}] &\leq  c_p\, (r\Delta_*)^{2(p - 1)}\int_{[0, r]^2}dtds \int_{[y_0, y_0 + \Delta_*]^2}dxdy \, \nonumber\\
 & \qquad \qquad \qquad \qquad\qquad \qquad\times \frac{\mbox{E}[|u(1_{]s, t]\times]y, x]})|^{(2p_0 - 1)p}] \, \|Du(t, x; s, y)\|_{\mathscr{H}}^p}{|t - s|^{(1 + 2p_0\gamma_1)p}|x - y|^{(1 + 2p_0\gamma_2)p}} \nonumber\\
 & \leq c_p \, (r\Delta_*)^{2(p - 1)}\int_{[0, r]^2}dtds \int_{[y_0, y_0 + \Delta_*]^2}dxdy\,\frac{|t - s|^{p_0p\theta_1}|x - y|^{p_0p\theta_2}}{|t - s|^{p(1 + 2p_0\gamma_1)}|x - y|^{p(1 + 2p_0\gamma_2)}}\nonumber \\
& \leq c_p \, (r\Delta_*)^{2p} \Delta_{\bullet}^{p(p_0\theta_1 - (1 + 2p_0\gamma_1))}\Delta_*^{p(p_0\theta_2 - (1 + 2p_0\gamma_2))}\nonumber \\
& \leq c_p \, r^{2p} \delta^{p(2p_0\theta_1 - 2(1 + 2p_0\gamma_1))}\delta^{p(p_0\theta_2 - (1 + 2p_0\gamma_2) + 2)}\nonumber \\
& = c_p \, r^{2p}\delta^{p(p_0(2\theta_1 + \theta_2) - 2p_0 (2\gamma_1 + \gamma_2) - 1)} = c_p \, r^{2p}\delta^{p(p_0 - \gamma_0)},
\end{align}
where the in the second inequality we use \eqref{eq1001}, and the derivation of the last equality follows the same reasoning as that of \eqref{eq2018-01-03-7}.

The proof of (c) is similar by using Lemma \ref{lemma13}, H\"{o}lder's inequality and \eqref{eq2016-06-14-1}, \eqref{eq1001} and \eqref{eq2017-09-28-2}.
\end{proof}

We proceed to introduce the random variables needed for Theorem \ref{th2017-08-24-2} to study the smoothness of densities of the random variables $F$ and $M_0$.
We define the function $\psi_0: \mathbb{R}^+ \rightarrow [0, 1]$ as an infinitely differentiable function such that
\begin{align} \label{eq2018-01-04-1}
\psi_0(x)
 =
\left\{\begin{array}{ll}
    0 & \hbox{if $x > 1$;} \\
   \psi_0(x) \in [0, 1]  & \hbox{if $x \in [\frac{1}{2}, 1]$,} \\
    1 & \hbox{if $x \leq \frac{1}{2}$.}
  \end{array}
\right.
\end{align}

We first introduce the random variables needed to prove the smoothness of density of $F$.
For $(z_1, z_2) \in \mathbb{R} \times ]0, \infty[$, set
\begin{align}\label{eq2017-09-28-4}
a = z_2/2 \quad \mbox{and} \quad A = \mathbb{R} \times ]a, \infty[.
\end{align}

Let $R = R(z_2, \delta_1)$ be defined as in Lemma \ref{lemma11}(a) for the specific value of $a$ in \eqref{eq2017-09-28-4}.
Define
\begin{align}\label{eq2017-10-08-1}
\psi(x):= \psi_0(x/R) \quad \mbox{so that} \quad \psi(x)
 =
\left\{\begin{array}{ll}
    0 & \hbox{if $x > R$;} \\
    \psi(x) \in [0, 1]  & \hbox{if $x \in [\frac{R}{2}, R]$,} \\
    1 & \hbox{if $x \leq \frac{R}{2}$}
  \end{array}
\right.
\end{align}
and
 \begin{align} \label{eq2017-11-26-1}
 \|\psi'\|_{\infty}:= \sup_{x \in \mathbb{R}}|\psi'(x)| \leq c \, R^{-1}
  \end{align}
  for a certain constant $c$ not depending on $z_2$.

If $I \times J \subset \, ]0, T] \times ]0, 1[$, let $c_1$, $C_1$, $c_2$, $C_2$ be as in  \eqref{eq2017-12-26-1} and \eqref{eq2017-12-26-3},  and $f_0: \mathbb{R} \mapsto [0, 1]$ be an infinitely differentiable function supported in $[c_1/2, (C_1 + T)/2]$ such that $f_0(t) = 1$, for all $t \in [c_1, C_1]$. Let $g_0: \mathbb{R} \mapsto [0, 1]$ be an infinitely differentiable function supported in $[c_2/2, (C_2 + 1)/2]$ such that $g_0(x) = 1$, for all $x \in [c_2, C_2]$. In the case of Neumann boundary conditions, if $I \subset \, ]0, T]$ and $y_0 = 0 \in J \subset [0, 1]$, we define $g_0$ to be an infinitely differentiable function with compact support such that $g_0(0) = 1$ and $g_0$ satisfies the same Neumann boundary conditions.

We define the $\mathscr{H}$-valued random variable $u_A^1$ evaluated at $(r, v)$ by
\begin{align}\label{eq2017-11-26-2}
u_A^1(r, v) &= \left(\frac{\partial}{\partial r} - \frac{\partial^2}{\partial v^2}\right)(f_0(r)g_0(v)).
\end{align}
In the case $I \times J \subset \, ]0, T] \times ]0, 1[$, from the choice of the functions $f_0$ and $g_0$, we see that there exists a constant $c$ such that for all $(s_0, y_0) \in I \times J$,
\begin{align}\label{eq2017-12-26-4}
\|u_A^1\|_{\mathscr{H}} \leq c.
\end{align}

Let $\phi_0 : \mathbb{R} \mapsto [0, 1]$ be an infinitely differentiable function supported in $[-1, 2]$ such that $\phi_0(v) = 1$, for all $v \in [0, 1]$.

For $y_0 \in J \subset [0, 1]$, we define $\phi_{\delta_1}$ as an infinitely differentiable function with compact support such that $\phi_{\delta_1}(y_0) = 1$ and satisfies the same boundary conditions at $0$ and $1$ as the Green kernel. In particular, if $J \subset \, ]0, 1[$ and $\delta_1$ satisfies the conditions in \eqref{eq2017-09-26}, then we choose the function $\phi_{\delta_1}$ in the following way:
\begin{align} \label{eq2017-11-26-3}
\phi_{\delta_1}(v) := \phi_0\left((v - y_0)/\delta_1^{1/2}\right), \quad v \in [0, 1],
\end{align}
so that, for some constant $c$,
\begin{align} \label{eq2017-11-26-4}
|\phi'_{\delta_1}(v)| \leq c\,\delta_1^{-1/2} \quad \mbox{and} \quad |\phi''_{\delta_1}(v)| \leq c\, \delta_1^{-1}, \quad \mbox{for all} \, \, v \in [0, 1].
\end{align}

Set
\begin{align}\label{eq2017-05-31-10000}
H(r, v) := \phi_{\delta_1}(v)\int_{s_0}^r\psi(Y_a)da, \quad (r, v) \in [s_0, s_0 + \delta_1] \times [0, 1].
\end{align}
We define the $\mathscr{H}$-valued random variable $u_A^2$ evaluated at $(r, v)$ by
\begin{align} \label{eq2017-11-26-5}
u_A^2(r, v)
 =
\left\{\begin{array}{ll}
    \left(\frac{\partial}{\partial r} - \frac{\partial^2}{\partial v^2}\right)H(r, v) & \hbox{if \, $(r, v) \in \, ]s_0, s_0 + \delta_1] \times [0, 1]$;} \\
   0 & \hbox{otherwise}.
  \end{array}
\right.
\end{align}
Finally, we define the random matrix $\gamma_A = (\gamma_A^{i, j})_{1 \leq i, j \leq 2}$ by
\begin{align} \label{eq2017-11-27-3}
\gamma_A = \left(
  \begin{array}{cc}
    1 & 0 \\
    0 & \int_{s_0}^{s_0 + \delta_1}\psi(Y_r)dr \\
  \end{array}
\right).
\end{align}

If $s_0 = 0 \in I \subset [0, T]$, we only consider the random variables $F_2$, $u_A^2$ and $\gamma_A^{2, 2}$ defined in \eqref{eq2017-12-19-2}, \eqref{eq2017-11-26-5} and \eqref{eq2017-11-27-3} with $s_0 = 0$, respectively.

We next introduce the random variables needed to prove the smoothness of density of $M_0$.
For $z \in \, ]0, \infty[$, set
\begin{align}\label{eq2017-09-28-411}
\bar{a} = z/2 \quad \mbox{and} \quad \bar{A} =  \, ]\bar{a}, \infty[.
\end{align}
Let $\bar{R} = \bar{R}(z, \delta)$ be defined as in Lemma \ref{lemma11}(b) for the specific value of $\bar{a}$ in \eqref{eq2017-09-28-411}.
Define
\begin{align}\label{eq2017-10-08-100}
\bar{\psi}(x):= \psi_0(x/\bar{R}) \quad \mbox{so that} \quad \bar{\psi}(x)
 =
\left\{\begin{array}{ll}
    0 & \hbox{if $x > \bar{R}$;} \\
    \bar{\psi}(x) \in [0, 1]  & \hbox{if $x \in [\frac{\bar{R}}{2}, \bar{R}]$,} \\
    1 & \hbox{if $x \leq \frac{\bar{R}}{2}$}
  \end{array}
\right.
\end{align}
and
 \begin{align} \label{eq2017-11-26-100}
 \|\bar{\psi}'\|_{\infty}:= \sup_{x \in \mathbb{R}}|\bar{\psi}'(x)| \leq c \, \bar{R}^{-1}
  \end{align}
  for a certain constant $c$ not depending on $z$.

 We define $\bar{\phi}_{\delta}$ as an infinitely differentiable function with compact support such that
 \begin{align} \label{eq2018-10-09-1}
 \bar{\phi}_{\delta}(v) = 1, \quad \mbox{for all} \, \,\,\, v \in [y_0, y_0 + \delta_2]
  \end{align}
  and satisfies the same boundary conditions at $0$ and $1$ as the Green kernel. In particular, if $J \subset \, ]0, 1[$ and $\delta_1$, $\delta_2$ satisfy the conditions in \eqref{eq2017-09-2600}, we choose the function $\bar{\phi}_{\delta}$ in the following way:
\begin{align} \label{eq2017-11-26-3000}
\bar{\phi}_{\delta}(v) := \phi_0\left((v - y_0)/\delta\right), \quad v \in [0, 1],
\end{align}
where the function $\phi_0$ is specified below \eqref{eq2017-12-26-4}, so that for some constant $c$,
\begin{align} \label{eq2017-11-26-4000}
|\bar{\phi}'_{\delta}(v)| \leq c\,\delta^{-1} \quad \mbox{and} \quad |\bar{\phi}''_{\delta}(v)| \leq c\, \delta^{-2}, \quad \mbox{for all} \, \, v \in [0, 1].
\end{align}

Set
\begin{align}\label{eq2017-05-31-1000011}
\bar{H}(r, v) := \bar{\phi}_{\delta}(v)\int_{0}^r\bar{\psi}(\bar{Y}_a)da, \quad (r, v) \in [0, \Delta_{\bullet}] \times [0, 1],
\end{align}
where $\{\bar{Y}_r: r \in [0, \Delta_{\bullet}]\}$ is defined in \eqref{eq2018-01-03-2}.
We define the $\mathscr{H}$-valued random variable $u_{\bar{A}}$ evaluated at $(r, v)$ by
\begin{align} \label{eq2017-11-26-511}
u_{\bar{A}}(r, v)
 =
\left\{\begin{array}{ll}
    \left(\frac{\partial}{\partial r} - \frac{\partial^2}{\partial v^2}\right)\bar{H}(r, v) & \hbox{if \, $(r, v) \in \, ]0, \Delta_{\bullet}] \times [0, 1]$;} \\
   0 & \hbox{otherwise}.
  \end{array}
\right.
\end{align}
Finally, we define the random variable
\begin{align} \label{eq2017-11-27-311}
\gamma_{\bar{A}} =  \int_{0}^{\Delta_{\bullet}}\bar{\psi}(\bar{Y}_r)dr.
 \end{align}

We now prove the smoothness of these random variables, as required in Theorem \ref{th2017-08-24-2}.

\begin{lemma}\label{lemma2017-1}
For $i, j \in \{1, 2\}$, $u_A^i \in \mathbb{D}^{\infty}(\mathscr{H})$, $\gamma_A^{i,j} \in \mathbb{D}^{\infty}$ and $u_{\bar{A}} \in \mathbb{D}^{\infty}(\mathscr{H})$, $\gamma_{\bar{A}} \in \mathbb{D}^{\infty}$.
\end{lemma}

\begin{proof}
We first prove that $\gamma_A^{2,2} \in \mathbb{D}^{\infty}$. Recall from  \eqref{eq2017-11-27-3} that $\gamma_A^{2,2} = \int_{s_0}^{s_0 + \delta_1}\psi(Y_r)dr$. Similar as in the proof of Lemma \ref{lemma13}, by exchanging the order of Malliavin derivative and integral,  we obtain that
\begin{align} \label{eq2017-11-28-1}
D\gamma_A^{2,2} = \int_{s_0}^{s_0 + \delta_1}\psi'(Y_r)DY_rdr.
\end{align}
In order to prove that $\gamma_A^{2,2} \in \mathbb{D}^{\infty}$, we can repeat this procedure and it remains to prove that for any $q, j \geq 1$,
\begin{align}\label{4eq32}
\sup_{r \in [s_0, s_0 + \delta_1]}\mbox{E}[\|D^j\psi(Y_r)\|_{\mathscr{H}^{\otimes j}}^q] < \infty.
\end{align}
To prove (\ref{4eq32}), by the Fa\`{a} di Bruno formula (see formula [24.1.2] in \cite{AlS79}), we find that
\begin{align} \label{newnewneweq33}
\|D^j\psi(Y_{r})\|_{\mathscr{H}^{\otimes j}} \leq c_j \, \prod_{i = 1}^j\|D^iY_{r}\|_{\mathscr{H}^{\otimes i}}^{l_i},
\end{align}
where $l_i$, $i = 1, \ldots, j$, are integers. \eqref{newnewneweq33} and \eqref{eq2017-11-26-9} imply \eqref{4eq32}. Hence $\gamma_A^{2,2}$ belongs to $\mathbb{D}^{\infty}$.

We can prove $\gamma_{\bar{A}} \in \mathbb{D}^{\infty}$ similarly by using \eqref{eq2017-11-26-9}.

We proceed to prove that $u_A^2 \in \mathbb{D}^{\infty}(\mathscr{H})$. By the definition of $u_A^2$ in \eqref{eq2017-11-26-5}, we can write
\begin{align} \label{eq2017-11-26-8}
u_A^2(r, v) &= \psi(Y_r)1_{]s_0, s_0 + \delta_1]}(r)\phi_{\delta_1}(v) - 1_{]s_0, s_0 + \delta_1]}(r)\phi_{\delta_1}''(v)\int_{s_0}^{r}\psi(Y_a)da \nonumber \\
 &:= u_A^{21}(r, v) -  u_A^{22}(r, v).
\end{align}
We first prove that $u_A^{21} \in \mathbb{D}^{\infty}(\mathscr{H})$. For $n \geq 1$, we define
\begin{align*}
Y_n^1(r, v) := \sum_{k = 1}^n\psi(Y_{s_0 + \delta_1k/n})1_{]s_0 + \delta_1(k - 1)/n, s_0 + \delta_1k/n]}(r)\phi_{\delta_1}(v).
\end{align*}
For almost every $(\omega, r, v) \in \Omega \times [0, T] \times [0, 1]$, $Y_n^1(r, v)$ converges to  $u_A^{21}(r, v)$ as $n \rightarrow \infty$. By the dominated convergence theorem, $Y_n^1$ converges to  $u_A^{21}$ in $L^p(\Omega; \mathscr{H})$ for any $p \geq 1$ as $n \rightarrow \infty$. Since for any $r \in [s_0, s_0 + \delta_1], Y_r \in \mathbb{D}^{\infty}$, by (\ref{4eq32}) and chain rule, we know that for any $r \in [s_0, s_0 + \delta_1], \psi(Y_r) \in \mathbb{D}^{\infty}$.

It is not difficult to check the following property: if $Z \in \mathbb{D}^{\infty}$ and $h \in \mathscr{H}$, then $Zh$ belongs to $\mathbb{D}^{\infty}(\mathscr{H})$.
Applying this property, we see that $Y_n^1$ belongs to $\mathbb{D}^{\infty}(\mathscr{H})$ and
\begin{align*}
DY_n^1(\cdot, *)= \sum_{k = 1}^nD\psi(Y_{s_0 + \delta_1k/n})1_{]s_0 + \delta_1(k - 1)/n, s_0 + \delta_1k/n]}(\cdot)\phi_{\delta_1}(*).
 \end{align*}
 For almost every $(\omega, r, v) \in \Omega \times [0, T] \times [0, 1]$, $DY_n^1(r, v)$ converges to $D\psi(Y_r)1_{]s_0, s_0 + \delta_1]}(r)\phi_{\delta_1}(v)$ as $n \rightarrow \infty$ since $r \mapsto D\psi(Y_r)$ is continuous. Moreover, by H\"{o}lder's inequality, for any $q \geq 1$,
\begin{align*}
\mbox{E}\left[\int_0^T\int_0^1\|DY_n^1(r, v)\|_{\mathscr{H}}^qdrdv\right]
&= \sum_{k = 1}^n\mbox{E}\left[\int_{s_0 + (k - 1)\delta_1/n}^{s_0 + k\delta_1/n}\|D\psi(Y_{s_0 + \delta_1k/n})\|_{\mathscr{H}}^qdr\right]\int_0^1\phi_{\delta_1}^q(v)dv\\
&\leq c \sum_{k = 1}^n\int_{s_0 + (k - 1)\delta_1/n}^{s_0 + k\delta_1/n} \sup_{r \in [s_0, s_0 + \delta_1]}\mbox{E}\left[\|DY_{r}\|_{\mathscr{H}}^q\right]dr\leq c',
\end{align*}
where the last inequality follows from (\ref{eq2017-11-26-9}). Applying \cite[Proposition 1.1]{ChW90} (with the measure space replaced by $(\Omega\times [0, T] \times [0, 1], \mbox{P}\times \lambda^2)$, where $\lambda^2$ is the Lebesgue measure on $[0, T] \times [0, 1]$), we have for any $q \geq 1$,
\begin{align*}
\lim_{n \rightarrow \infty}\mbox{E}\left[\int_0^T\int_0^1\|DY^1_n(r, v) - D\psi(Y_r)1_{]s_0, s_0 + \delta_1]}(r)\phi_{\delta_1}(v)\|_{\mathscr{H}}^qdrdv\right] = 0,
\end{align*}
which implies
\begin{align*}
\lim_{n \rightarrow \infty}\mbox{E}\left[\left(\int_0^T\int_0^1\|DY^1_n(r, v) - D\psi(Y_r)1_{]s_0, s_0 + \delta_1]}(r)\phi_{\delta_1}(v)\|_{\mathscr{H}}^2drdv\right)^{q/2}\right] = 0.
\end{align*}
Thus for any $q \geq 1$, $DY^1_n(\cdot, *)$ converges to $D\psi(Y_{\cdot})1_{]s_0, s_0 + \delta_1]}(\cdot)\phi_{\delta_1}(*)$ in $L^q(\Omega, \mathscr{H}^{\otimes 2})$ as $n \rightarrow \infty$. Since $D$ is closable, we obtain
\begin{align*}
Du_A^{21}(\cdot, *) = D\psi(Y_{\cdot})1_{]s_0, s_0 + \delta_1]}(\cdot)\phi_{\delta_1}(*).
\end{align*}
We repeat this procedure and apply (\ref{4eq32}) to conclude $u_A^{21} \in \mathbb{D}^{\infty}(\mathscr{H})$.

The proof for $u_A^{22} \in \mathbb{D}^{\infty}(\mathscr{H})$ is similar. We discretize $u_A^{22}(r, v)$ by
\begin{align*}
Y_n^2(r, v) := \sum_{k = 1}^n\int_{s_0}^{s_0 + k\delta_1/n}\psi(Y_{a})da \, 1_{]s_0 + (k - 1)\delta_1/n, s_0 + k\delta_1/n]}(r)\phi_{\delta_1}''(v).
\end{align*}
In fact, the proof of $\gamma_A^{2, 2} \in \mathbb{D}^{\infty}$ indicates that for any $r \in [s_0, s_0 + \delta_1]$, $\int_{s_0}^r\psi(Y_a)da \in \mathbb{D}^\infty$. Hence applying the property again, we see that $Y_n^2 \in \mathbb{D}^{\infty}(\mathscr{H})$. Similarly, we apply (\ref{4eq32}) again to conclude $u_A^{22} \in \mathbb{D}^{\infty}(\mathscr{H})$.

The proof of $u_{\bar{A}} \in \mathbb{D}^{\infty}(\mathscr{H})$ is similar.
\end{proof}

The following results give some estimates on the $L^p(\Omega)$-norm of $(\gamma_A^{2,2})^{-1}$ and $\gamma_{\bar{A}}^{-1}$.
\begin{lemma} \label{lemma2017-05-12-1}
\begin{itemize}
  \item [(a)]
  The random variable $\gamma_A^{2,2}$ has finite negative moments of all orders. Furthermore,
for any $p \geq 1$, there exists a constant $c_p$, not depending on $(s_0, y_0) \in I \times J$, such that for all small $\delta_1 > 0$ and for $z_2 \geq \delta_1^{1/4}$,
\begin{align} \label{eq2017-6}
\|(\gamma_A^{2,2})^{-1}\|_{L^p(\Omega)} \leq c_p \, \delta_1^{-1}.
\end{align}
\item [(b)]The random variable $\gamma_{\bar{A}}$ has finite negative moments of all orders. Furthermore,
for any $p \geq 1$, there exists a constant $c_p$, not depending on $y_0 \in J$, such that for all small $\delta_1, \, \delta_2 > 0$ and for $z \geq (\delta_1^{1/2} + \delta_2)^{1/2}$,
\begin{align}\label{eq2018-01-02-14}
\|\gamma_{\bar{A}}^{-1}\|_{L^p(\Omega)} \leq c_p \, (\delta_1^{1/2} + \delta_2)^{-2}.
\end{align}
\end{itemize}
\end{lemma}
\begin{proof}
We first prove (a). By the definition of the function $\psi$ in \eqref{eq2017-10-08-1},
\begin{align*}
\gamma_A^{2,2} \geq \int_{s_0}^{s_0 + \delta_1} 1_{\{Y_r \leq \frac{R}{2}\}}dr := \bar{X}.
\end{align*}
For $\epsilon < \delta_1$ and any $q \geq 1$, since $r \mapsto Y_r$ is increasing, we have
\begin{align}
\mbox{P}\{\bar{X} < \epsilon\}
&\leq \mbox{P}\{Y_{s_0 + \epsilon} \geq R/2\} \leq (2/R)^q \mbox{E}[|Y_{s_0 + \epsilon}|^q] \leq c_q \, R^{-q}\epsilon^{2q}\delta_1^{(p_0 - \gamma_0)q/2},
\end{align}
where in the second inequality we use Markov's inequality, and the last inequality is because of \eqref{eq2018-01-03-6}.
This shows that the random variable $\gamma_A^{2,2}$ has finite negative moments of all orders by \cite[Chapter 3, Lemma 4.4]{DKMNXR08}. Moreover, for any $p \geq 1$ and $q > p/2$,
\begin{align*}
\mbox{E}[\bar{X}^{-p}] &= p\int_0^{\infty}y^{p - 1}\mbox{P}(\bar{X}^{-1} > y)dy  \nonumber\\
&= p\int_0^{\delta_1^{-1}}y^{p - 1}\mbox{P}(\bar{X}^{-1} > y)dy + p\int_{\delta_1^{-1}}^{\infty}y^{p - 1}\mbox{P}(\bar{X}^{-1} > y)dy\nonumber\\
&\leq c \, \delta_1^{-p} + cR^{-q} \delta_1^{(p_0 - \gamma_0)q/2} \int_{\delta_1^{-1}}^{\infty}y^{p - 1}y^{-2q}dy \\
& = c \, \delta_1^{-p} + cR^{-q} \delta_1^{(p_0 - \gamma_0 + 4)q/2 - p}.
\end{align*}
Using the definition of $R$ in \eqref{eq2017-10-12-5}, this is equal to
\begin{align*}
c\,  \delta_1^{-p}\left(1 + a^{-2p_0q}\delta_1^{(\gamma_0 - 4)q/2}\delta_1^{(p_0 - \gamma_0 + 4)q/2}\right).
\end{align*}
Under the assumption $z_2 \geq \delta_1^{1/4}$, by \eqref{eq2017-09-28-4}, this is bounded above by
\begin{align*}
c\,  \delta_1^{-p}(1 + \delta_1^{-\frac{1}{4} \times 2p_0q}\delta_1^{(\gamma_0 - 4)q/2}\delta_1^{(p_0 - \gamma_0 + 4)q/2}) = 2c\,  \delta_1^{-p},
\end{align*}
which implies (\ref{eq2017-6}).

We proceed to prove (b). Similarly, by the definition of the function $\bar{\psi}$ in \eqref{eq2017-10-08-100},
\begin{align*}
\gamma_{\bar{A}} \geq \int_0^{\Delta_{\bullet}}1_{\{\bar{Y}_r \leq \frac{\bar{R}}{2}\}}dr:= \tilde{X}.
\end{align*}
For any $0 < \epsilon < \Delta_{\bullet}$, since $r \mapsto \bar{Y}_r$ is increasing,
\begin{align}
\mbox{P}\{\tilde{X} < \epsilon \} &\leq \mbox{P}\{\bar{Y}_{\epsilon} \geq \bar{R}/2\}  \leq (2/\bar{R})^{q}\mbox{E}[\bar{Y}_{\epsilon}^q] \leq c_q \bar{R}^{-q}\, \epsilon^{2q}\delta^{(p_0 - \gamma_0)q},
\end{align}
where, in the last inequality, we use \eqref{eq2018-01-03-8}. Hence the random variable $\gamma_{\bar{A}}$ has finite negative moments of all orders by \cite[Chapter 3, Lemma 4.4]{DKMNXR08}. Moreover, for any $p \geq 1$ and $q > p/2$,
\begin{align*}
\mbox{E}[\tilde{X}^{-p}] &= p\int_0^{\infty}y^{p - 1}\mbox{P}(\tilde{X}^{-1} > y)dy  \nonumber\\
&= p\int_0^{\Delta_{\bullet}^{-1}}y^{p - 1}\mbox{P}(\tilde{X}^{-1} > y)dy + p\int_{\Delta_{\bullet}^{-1}}^{\infty}y^{p - 1}\mbox{P}(\tilde{X}^{-1} > y)dy\nonumber\\
&\leq c \, \Delta_{\bullet}^{-p} + c\bar{R}^{-q}\delta^{(p_0 - \gamma_0)q}\int_{\Delta_{\bullet}^{-1}}^{\infty}y^{p - 1}y^{-2q}dy \\
& = c \, \Delta_{\bullet}^{-p} + c\bar{R}^{-q}\delta^{(p_0 - \gamma_0)q} \Delta_{\bullet}^{2q - p}.
\end{align*}
Using the definition of $\bar{R}$ in \eqref{eq2017-10-12-500}, this is equal to
\begin{align*}
c\,  \Delta_{\bullet}^{-p}\left(1 + \bar{a}^{-2p_0q}\delta^{-(4 -\gamma_0)q}\delta^{(p_0 - \gamma_0)q} \Delta_{\bullet}^{2q}\right).
\end{align*}
Under the assumption $z \geq \delta^{1/2} = (\delta_1^{1/2} + \delta_2)^{1/2}$, by \eqref{eq2017-09-28-411} and \eqref{eq2018-01-03-1}, this is bounded above by
\begin{align*}
c\,  \Delta_{\bullet}^{-p}(1 + \delta^{-\frac{1}{2} \times 2p_0q}\delta^{-(4 -\gamma_0)q}\delta^{(p_0 - \gamma_0)q} \Delta_{\bullet}^{2q}) = 2c\,  \Delta_{\bullet}^{-p},
\end{align*}
which implies \eqref{eq2018-01-02-14}.
\end{proof}

Now we are ready to verify that the random variables introduced above satisfy the condition (iii) of Theorem \ref{th2017-08-24-2}.

\begin{lemma}\label{lemma2017-2}
\begin{itemize}
  \item [(a)]Let $A$ be the set defined in \eqref{eq2017-09-28-4}.  On the event $\{F \in A\} = \{F_2 > a\}$,
$\langle DF_i, u_A^j\rangle_{\mathscr{H}} = \gamma_A^{i, j}$ for $i, j \in \{1, 2\}.$
  \item [(b)]Let $\bar{A}$ be the set defined in \eqref{eq2017-09-28-411}.  On the event $\{M_0 \in \bar{A}\} = \{M_0 > \bar{a}\}$, $\langle DM_0, u_{\bar{A}} \rangle_{\mathscr{H}} = \gamma_{\bar{A}}$.
\end{itemize}

\end{lemma}
\begin{proof}
We first prove (a).
If $s_0 > 0$, by the definition of $u_A^1$ in \eqref{eq2017-11-26-2} and of the functions $f_0$, $g_0$, and by Lemma \ref{lemma2017-11-27-1}, we have
\begin{align}\label{eq2017-08-28-1}
\langle DF_1, u_A^1\rangle_{\mathscr{H}} &= \int_0^{s_0}\int_0^1G(s_0 - r, y_0, v)\left(\frac{\partial}{\partial r} - \frac{\partial^2}{\partial v^2}\right)(f_0(r)g_0(v))drdv\\
&= f_0(s_0)g_0(y_0) =1 = \gamma_A^{1, 1}. \nonumber
\end{align}
Second, from the definition of $u_A^2$ in \eqref{eq2017-11-26-5}, it is obvious that
\begin{align} \label{eq2017-08-28-2}
\langle DF_1, u_A^2\rangle_{\mathscr{H}} &=\int_0^{s_0}\int_0^1G(s_0 - r, y_0, v)
\left(\frac{\partial}{\partial r} - \frac{\partial^2}{\partial v^2}\right)H(r, v)1_{\{s_0 < r \leq s_0 + \delta_1\}}drdv =0.
\end{align}
By Lemmas \ref{lemma9} and  \ref{lemma2017-11-27-1},
\begin{align}\label{eq2017-08-28-3}
\langle DF_2, u_A^1\rangle_{\mathscr{H}} &= \int_0^{S}\int_0^1G(S - r, y_0, v)\left(\frac{\partial}{\partial r} - \frac{\partial^2}{\partial v^2}\right)(f_0(r)g_0(v))drdv \nonumber\\
&\quad - \int_0^{s_0}\int_0^1G(s_0 - r, y_0, v)\left(\frac{\partial}{\partial r} - \frac{\partial^2}{\partial v^2}\right)(f_0(r)g_0(v))drdv\\
&= f_0(S)g_0(y_0) - f_0(s_0)g_0(y_0) = 1 - 1=0. \nonumber
\end{align}
Furthermore, by Lemmas \ref{lemma9} and \ref{lemma2017-11-27-1}, for both cases $s_0 > 0$ and $s_0 = 0$,
\begin{align}\label{eq2017-08-28-4}
\langle DF_2, u_A^2\rangle_{\mathscr{H}} &= \int_0^Sdr\int_0^1dv \, G(S - r, y_0, v)u_A^2(r, v) - \int_0^{s_0}dr\int_0^1dv\, G(s_0 - r, y_0, v)u_A^2(r, v)\nonumber\\
&= \int_{s_0}^Sdr\int_0^1dv \, G(S - r, y_0, v)\left(\frac{\partial}{\partial r} - \frac{\partial^2}{\partial v^2}\right)H(r, v) - 0 \nonumber \\
&= \int_{0}^{S - s_0}dr\int_0^1dv \, G(S - s_0 - r, y_0, v)\left(\frac{\partial}{\partial r} - \frac{\partial^2}{\partial v^2}\right)H(s_0 + r, v) \nonumber \\
& = H(S, y_0).
\end{align}
Therefore,
\begin{align} \label{eq2017-10-02-1}
\langle DF_2, u_A^2\rangle_{\mathscr{H}} &= \phi_{\delta_1}(y_0)\int_{s_0}^S\psi(Y_r)dr= \int_{s_0}^S\psi(Y_r)dr,
\end{align}
where, in the second equality, we use the fact that $\phi_{\delta_1}(y_0) = 1$. Moreover, on the event $\{F \in A\} = \{F_2 > a\}$, we observe that if $r > S \geq s_0$, then $\psi(Y_r) = 0$. Otherwise, we would have $\psi(Y_r) > 0$, hence $Y_r \leq R$ for some $r > S$, and by Lemma \ref{lemma11}(a), this implies that
\begin{align*}
F_2 = \bar{u}(S, y_0) = \sup_{t \in [s_0, r]}\bar{u}(t, y_0) \leq a < F_2,
 \end{align*}
 which is a contradiction. Hence, on $\{F \in A\} = \{F_2 > a\}$,  the last integral in \eqref{eq2017-10-02-1} is equal to
$
\int_{s_0}^{s_0 + \delta_1}\psi(Y_r)dr = \gamma_A^{2, 2}.
$
This completes the proof of (a).

We now prove (b). By Lemma \ref{lemma9},
\begin{align}\label{eq2018-01-02-13}
\langle DM_0, u_{\bar{A}} \rangle_{\mathscr{H}} & = \int_0^{\bar{S}}\int_0^1 G(\bar{S} - r, \bar{X}, v)\left(\frac{\partial}{\partial r} - \frac{\partial^2}{\partial v^2}\right)\bar{H}(r, v) dvdr \nonumber \\
& = \bar{H}(\bar{S}, \bar{X}) = \bar{\phi}_{\delta}(\bar{X})\int_0^{\bar{S}}\bar{\psi}(\bar{Y}_r)dr.
\end{align}
Since $\bar{X} \in [y_0, y_0 + \delta_2]$, the definition of the function $\bar{\phi}_{\delta}$ in \eqref{eq2018-10-09-1} implies that $\bar{\phi}_{\delta}(\bar{X}) \equiv 1$. Hence,
\begin{align*}
\langle DM_0, u_{\bar{A}} \rangle_{\mathscr{H}} = \int_0^{\bar{S}}\bar{\psi}(\bar{Y}_r)dr.
\end{align*}
On the event $\{M_0 > \bar{a}\}$, for $r > \bar{S}$, we have $\bar{\psi}(\bar{Y}_r) = 0$. Otherwise, we would have $\bar{\psi}(\bar{Y}_r) > 0$, hence $\bar{Y}_r \leq \bar{R}$ and by Lemma \ref{lemma11}(b), this implies that
 \begin{align*}
 M_0 = u(\bar{S}, \bar{X}) = \sup_{(t, x) \in [0, r] \times [y_0, y_0 + \delta_2]}u(t, x) \leq \bar{a} < M_0,
 \end{align*}
which is a contradiction. Therefore, on the event $\{M_0 \in \bar{A}\}$,
\begin{align*}
\langle DM_0, u_{\bar{A}} \rangle_{\mathscr{H}} = \int_0^{\Delta_{\bullet}}\bar{\psi}(\bar{Y}_r)dr = \gamma_{\bar{A}}.
\end{align*}
This proves (b).
\end{proof}

\begin{proof}[Proof of Theorem \ref{the2017-11-25-1}]
(a) The strict positivity of $F_2$ has been proved in Lemma \ref{lemma100}(a). For $(s_0, y_0) \in I \times J \subset [0, T] \times [0, 1]$ with $s_0 > 0$, by Lemmas \ref{lemma2017-1},  \ref{lemma2017-05-12-1}(a), \ref{lemma2017-2}(a) and Theorem \ref{th2017-08-24-2}, the random vector $F$ has an infinitely differentiable density on $\mathbb{R} \times ]z_2/2, \infty[$. Since the choice of $z_2$ is arbitrary, the random vector $F$ possesses an infinitely differentiable density on $\mathbb{R} \times ]0, \infty[$. Using the same argument, if $s_0 = 0$, then the random variable $F_2$ has an infinitely differentiable density on $ ]0, \infty[$.

(b) The strict positivity of $M_0$ has been proved in Lemma \ref{lemma100}(b).
The proof of smoothness of the density of $M_0$ is similar to that of Theorem \ref{the2017-11-25-1}(a) by using Lemmas \ref{lemma2017-1},  \ref{lemma2017-05-12-1}(b), \ref{lemma2017-2}(b) and Theorem \ref{th2017-08-24-2}.
\end{proof}

\section{Formulas for the densities of $F$ and $M_0$}
We now derive the expression for the probability density functions of $F$ and $M_0$ from the integration by parts formula; see \cite[(2.25)]{Nua06}.
\begin{prop} \label{prop2017-10-02-1}
\begin{itemize}
  \item [(a)]The probability density function of $F$ at $(z_1, z_2) \in \mathbb{R} \times ]0, \infty[$ is given by
\begin{align} \label{eq2017-2}
p(z_1, z_2) = \mbox{E}\left[1_{\{F_1 > z_1, F_2 > z_2\}}\delta\left(u_A^1\delta\left(u_A^2/\gamma_A^{2, 2}\right)\right)\right].
\end{align}
  \item [(b)] The probability density function of $M_0$ at $z \in \, ]0, \infty[$ is given by
  \begin{align}\label{eq2018-01-02-15}
p_0(z) = \mbox{E}[1_{\{M_0 > z\}}\delta(u_{\bar{A}}/\gamma_{\bar{A}})].
\end{align}
\end{itemize}

\end{prop}
\begin{proof}
We first derive the formula \eqref{eq2018-01-02-15}. Let $\tilde{\kappa}_{z} : \mathbb{R} \mapsto [0, 1]$ be an infinitely differentiable function such that $\tilde{\kappa}_{z}(x) = 0$ for all $x \leq \frac{2z}{3}$ and $\tilde{\kappa}_{z}(x) = 1$ for all $x \geq \frac{3z}{4}$. Define $G_0 = \tilde{\kappa}(M_0)$. Consider $\bar{a}$ and $\bar{A}$ as in \eqref{eq2017-09-28-411}. It is clear that on the set $\{M_0 \not\in \bar{A}\}$, we have $G_0 = 0$.

Let $f$ be a function in the space $C_0^{\infty}(\mathbb{R})$ of infinitely differentiable functions with compact support. Set $\varphi(x) = \int_{-\infty}^xf(y)dy$. On $\{M_0 \in \bar{A}\}$, by the chain rule of Malliavin calculus (see \cite[Proposition 1.2.3]{Nua06}) and Lemma \ref{lemma2017-2}(b), we have
\begin{align*}
\langle D\varphi(M_0), u_{\bar{A}}\rangle_{\mathscr{H}} = \varphi'(M_0)\langle DM_0, u_{\bar{A}} \rangle_{\mathscr{H}} = \varphi'(M_0)\gamma_{\bar{A}}.
\end{align*}
Hence,
$
\varphi'(M_0)= \langle D\varphi(M_0), u_{\bar{A}}/\gamma_{\bar{A}} \rangle_{\mathscr{H}}.
$
Since $G_0 = 0$ on the set $\{M_0 \not\in \bar{A}\}$, we obtain
\begin{align*}
G_0\varphi'(M_0)= G_0\langle D\varphi(M_0), u_{\bar{A}}/\gamma_{\bar{A}} \rangle_{\mathscr{H}}.
\end{align*}
Taking expectations on both sides of the above equation and using the duality relationship between the derivative and the divergence operators we get
\begin{align} \label{eq2017-3-02-26}
\mbox{E}[G_0\varphi'(M_0)] = \mbox{E}[\varphi(M_0) \delta(G_0u_{\bar{A}}/\gamma_{\bar{A}})].
\end{align}
Using the fact that
\begin{align} \label{eq2017-11-27-600}
\varphi(M_0) = \int_{-\infty}^{M_0}\varphi'(y)dy
\end{align}
and Fubini's theorem, we obtain that
 \begin{align}
\mbox{E}[G_0\varphi'(M_0)] = \int_{\mathbb{R}}\varphi'(y)\mbox{E}[1_{\{M_0 > y\}}\delta(G_0u_{\bar{A}}/\gamma_{\bar{A}})]dy,
\end{align}
and equivalently,
\begin{align}
\mbox{E}[G_0f(M_0)] = \int_{\mathbb{R}}f(y)\mbox{E}[1_{\{M_0 > y\}}\delta(G_0u_{\bar{A}}/\gamma_{\bar{A}})]dy.
\end{align}
Since $G_0 = 1$ on the set $\{M_0 \geq  \frac{3z}{4}\}$, this implies that for any $y \in  \, ]\frac{3z}{4}, \infty[$, the density function of $M_0$ at $y$ is given by
$
p_0(y) = \mbox{E}[1_{\{M_0 > y\}}\delta(G_0u_{\bar{A}}/\gamma_{\bar{A}})].
$
In particular,
\begin{align*}
p_0(z) = \mbox{E}[1_{\{M_0 > z\}}\delta(G_0u_{\bar{A}}/\gamma_{\bar{A}})].
\end{align*}
Since $G_0 = 1$ on the set $\{M_0 > z\}$, by the local property of $\delta$ (see \cite[Proposition 1.3.15]{Nua06}), we obtain formula \eqref{eq2018-01-02-15}.

We now derive the formula \eqref{eq2017-2}.
Let $\kappa_{z_2} : \mathbb{R} \mapsto [0, 1]$ be an infinitely differentiable function such that $\kappa_{z_2}(x) = 0$ for all $x \leq \frac{2z_2}{3}$ and $\kappa_{z_2}(x) = 1$ for all $x \geq \frac{3z_2}{4}$. Define $\bar{\kappa}(y_1, y_2) = \kappa_{z_2}(y_2)$ and $G = \bar{\kappa}(F)$. Consider $a$ and $A$ as in \eqref{eq2017-09-28-4}. It is clear that on the set $\{F \not\in A\}$, we have $G = 0$.

Let $g$ be a function in the space $C_0^{\infty}(\mathbb{R}^2)$ of infinitely differentiable functions with compact support. Set
\begin{align}
\varphi(x_1, x_2) = \int_{-\infty}^{x_1}\int_{-\infty}^{x_2}g(y_1, y_2)dy_1dy_2.
\end{align}
On $\{F \in A\}$, by the chain rule of Malliavin calculus (see \cite[Proposition 1.2.3]{Nua06}) and Lemma \ref{lemma2017-2}(a), we have
\begin{align*}
\langle D\partial_1\varphi(F), u_A^j\rangle_{\mathscr{H}} = \sum_{i = 1}^2 \partial_{1i}\varphi(F)\langle DF_i, u_A^j \rangle_{\mathscr{H}} = \sum_{i = 1}^2 \partial_{1i}\varphi(F)\gamma_A^{i,j},
\end{align*}
where the notation $\partial_{1i}$ means we take the partial derivative with respect to the first variable and then take the partial derivative with respect to the $i$th variable.
Consequently,
$
\partial_{12}\varphi(F) = \sum_{k = 1}^2\langle D\partial_1\varphi(F), u_A^k\rangle_{\mathscr{H}}(\gamma_A^{-1})^{k, 2}.
$
Since $G = 0$ on the set $\{F \not\in A\}$, we obtain
\begin{align*}
G\partial_{12}\varphi(F) = \sum_{k = 1}^2G\langle D\partial_1\varphi(F), u_A^k\rangle_{\mathscr{H}}(\gamma_A^{-1})^{k, 2}.
\end{align*}
Taking expectations on both sides of the above equation and using the duality relationship between the derivative and the divergence operators we get
\begin{align} \label{eq2017-3}
\mbox{E}[G\partial_{12}\varphi(F)] = \mbox{E}[\partial_1\varphi(F) \delta(\sum_{k = 1}^2Gu_A^k(\gamma_A^{-1})^{k, 2})].
\end{align}
We denote
$
\bar{G} = \delta(\sum_{k = 1}^2Gu_A^k(\gamma_A^{-1})^{k, 2}).
$
Since  $G = 0$ on the set $\{F \not\in A\}$, the local property of $\delta$ (see \cite[Proposition 1.3.15]{Nua06}) implies that $\bar{G} = 0$ on the set $\{F \not\in A\}$. On the other hand, on $\{F \in A\}$,  by the chain rule and Lemma \ref{lemma2017-2},
\begin{align*}
\langle D\varphi(F), u_A^j\rangle_{\mathscr{H}} = \sum_{i = 1}^2 \partial_{i}\varphi(F)\langle DF_i, u_A^j \rangle_{\mathscr{H}} = \sum_{i = 1}^2 \partial_{i}\varphi(F)\gamma_A^{i,j},
\end{align*}
which implies that on $\{F \in A\}$,
$
\partial_{1}\varphi(F) = \sum_{n = 1}^2\langle D\varphi(F), u_A^n\rangle_{\mathscr{H}}(\gamma_A^{-1})^{n, 1}.
$
Multiplying both sides of the above equality by $\bar{G}$, we obtain
\begin{align}\label{eq2017-4}
\bar{G}\partial_{1}\varphi(F) = \sum_{n = 1}^2\bar{G}\langle D\varphi(F), u_A^n\rangle_{\mathscr{H}}(\gamma_A^{-1})^{n, 1}.
\end{align}
We substitute (\ref{eq2017-4}) into (\ref{eq2017-3}) and we obtain
\begin{align}
\mbox{E}[G\partial_{12}\varphi(F)] &= \mbox{E}\left[\sum_{n = 1}^2\bar{G}\langle D\varphi(F), u_A^n\rangle_{\mathscr{H}}(\gamma_A^{-1})^{n, 1}\right] \nonumber \\
&= \mbox{E}\left[\varphi(F)\delta\left(\sum_{n = 1}^2\bar{G}u_A^n(\gamma_A^{-1})^{n, 1}\right)\right] \nonumber\\
&= \mbox{E}\left[\varphi(F)\delta\left(\sum_{n = 1}^2\delta\left(\sum_{k = 1}^2Gu_A^k(\gamma_A^{-1})^{k, 2}\right)u_A^n(\gamma_A^{-1})^{n, 1}\right)\right]. \nonumber
\end{align}
Since $(\gamma_A^{-1})^{1, 1}=1$ and $(\gamma_A^{-1})^{1, 2} = (\gamma_A^{-1})^{2, 1} = 0$ by \eqref{eq2017-11-27-3}, this is equal to
\begin{align}
 \mbox{E}\left[\varphi(F)\delta(\delta(Gu_A^2(\gamma_A^{-1})^{2, 2})u_A^1(\gamma_A^{-1})^{1, 1})\right] = \mbox{E}[\varphi(F)\delta(\delta(Gu_A^2/\gamma_A^{2, 2})u_A^1)]. \nonumber
 \end{align}
Using the fact that
\begin{align} \label{eq2017-11-27-6}
\varphi(F) = \int_{-\infty}^{F_1}\int_{-\infty}^{F_2}\partial_{12}\varphi(y_1, y_2)dy_1dy_2
\end{align}
and Fubini's theorem, we obtain that
 \begin{align}
\mbox{E}[G\partial_{12}\varphi(F)] = \int_{\mathbb{R}^2}\partial_{12}\varphi(y_1, y_2)\mbox{E}[1_{\{F_1 > y_1, F_2 > y_2\}}\delta(\delta(Gu_A^2/\gamma_A^{2, 2})u_A^1)]dy_1dy_2,
\end{align}
and equivalently,
 \begin{align}
\mbox{E}[Gg(F)] = \int_{\mathbb{R}^2}g(y_1, y_2)\mbox{E}[1_{\{F_1 > y_1, F_2 > y_2\}}\delta(\delta(Gu_A^2/\gamma_A^{2, 2})u_A^1)]dy_1dy_2.
\end{align}

Since $G = 1$ on the set $\{F \in \mathbb{R} \times [\frac{3z_2}{4}, \infty[\}$, this implies that for any $(y_1, y_2) \in \mathbb{R} \times ]\frac{3z_2}{4}, \infty[$, the density function of $F$ at $(y_1, y_2)$ is given by
\begin{align*}
p(y_1, y_2) = \mbox{E}[1_{\{F_1 > y_1, F_2 > y_2\}}\delta(\delta(Gu_A^2/\gamma_A^{2, 2})u_A^1)].
\end{align*}
In particular,
\begin{align*}
p(z_1, z_2) = \mbox{E}[1_{\{F_1 > z_1, F_2 > z_2\}}\delta(\delta(Gu_A^2/\gamma_A^{2, 2})u_A^1)].
\end{align*}
Since $G = 1$ on the set $\{F_2 > z_2\}$, by the local property of $\delta$ (see \cite[Proposition 1.3.15]{Nua06}), we obtain
formula \eqref{eq2017-2}
\end{proof}

\begin{remark}\label{remark2017-10-11-1}
In the proof of Proposition \ref{prop2017-10-02-1}, if we use the fact that
\begin{align*}
\varphi(F) = -\int_{F_1}^{+ \infty}\int_{- \infty}^{F_2}\partial_{12}\varphi(y_1, y_2)dy_2dy_1
\end{align*}
instead of \eqref{eq2017-11-27-6}, we obtain another formula for the joint density:
\begin{align}\label{eq2017-09-07-10}
p(z_1, z_2) = -\mbox{E}[1_{\{F_1 < z_1, F_2 > z_2\}}\delta(\delta(u_A^2/\gamma_A^{2, 2})u_A^1)].
\end{align}
\end{remark}

\section{Gaussian-type upper bound on the density of $F$} \label{section4.4}

In this section, we fix $I \times J \subset \, ]0, T] \times ]0, 1[$ and assume that $\delta_1$ satisfies the conditions in \eqref{eq2017-09-26}.
We derive an estimate on the density of $F$ from the formula obtained in  Proposition \ref{prop2017-10-02-1}. This estimate will establish Theorem \ref{theorem100}.

First, from \eqref{eq2017-2}  and applying H\"{o}lder's inequality, for $z_1 \geq 0$,
\begin{align}\label{eq2017-09-07-11}
p(z_1, z_2) \leq  \mbox{P}\{F_1 > z_1\}^{1/4}\, \mbox{P}\{F_2 > z_2\}^{1/4}\, \|\delta(\delta(u_A^2/\gamma_A^{2, 2})u_A^1)\|_{L^2(\Omega)}.
\end{align}
On the other hand, if $z_1 < 0$, applying H\"{o}lder's inequality to (\ref{eq2017-09-07-10}), we obtain
\begin{align}\label{eq2017-09-07-12}
p(z_1, z_2) \leq  \mbox{P}\{F_1 < z_1\}^{1/4}\, \mbox{P}\{F_2 > z_2\}^{1/4}\, \|\delta(\delta(u_A^2/\gamma_A^{2, 2})u_A^1)\|_{L^2(\Omega)}.
\end{align}
Combining (\ref{eq2017-09-07-11}) and (\ref{eq2017-09-07-12}), we obtain that, for all $(z_1, z_2) \in \mathbb{R} \times ]0, \infty[$,
\begin{align}\label{eq2017-09-07-13}
p(z_1, z_2) \leq  \mbox{P}\{|F_1| > |z_1|\}^{1/4}\, \mbox{P}\{F_2 > z_2\}^{1/4}\, \|\delta(\delta(u_A^2/\gamma_A^{2, 2})u_A^1)\|_{L^2(\Omega)}.
\end{align}

In what follows, we use the properties of the Skorohod integral $\delta$ to express $\delta(\delta(u_A^2/\gamma_A^{2, 2})u_A^1)$.

\begin{lemma}\label{lemma2017-12-17-1}
 \begin{align}\label{eq2017-08-28-10}
\delta(\delta(u_A^2/\gamma_A^{2, 2})u_A^1)  &= T_1 + T_2 - T_3 + T_4 - T_5 + T_6,
\end{align}
where
\begin{align}
T_1 & = (\gamma_A^{2, 2})^{-1}\delta(u_A^2)\delta(u_A^1), \, \, T_2 =(\gamma_A^{2, 2})^{-2} \langle D\gamma_A^{2, 2}, u_A^2 \rangle_{\mathscr{H}}\delta(u_A^1),  \,\, T_3  = (\gamma_A^{2, 2})^{-1}\langle D\delta(u_A^2), u_A^1 \rangle_{\mathscr{H}}, \nonumber \\
 T_4 & = (\gamma_A^{2, 2})^{-2}\delta(u_A^2) \langle D\gamma_A^{2, 2}, u_A^1 \rangle_{\mathscr{H}}, \quad
T_5 = 2(\gamma_A^{2,2})^{-3}\langle D\gamma_A^{2, 2}, u_A^2 \rangle_{\mathscr{H}}\langle D\gamma_A^{2, 2}, u_A^1 \rangle_{\mathscr{H}}, \nonumber \\
T_6 & = (\gamma_A^{2, 2})^{-2}\langle D\langle D\gamma_A^{2, 2}, u_A^2 \rangle_{\mathscr{H}}, u_A^1 \rangle_{\mathscr{H}}. \nonumber
\end{align}
 \end{lemma}
 \begin{proof}
 First, by \cite[(1.48)]{Nua06},
\begin{align}\label{eq2017-08-28-5}
\delta(\delta(u_A^2/\gamma_A^{2, 2})u_A^1) &= \delta(u_A^2/\gamma_A^{2, 2})\delta(u_A^1) - \langle D\delta(u_A^2/\gamma_A^{2, 2}), u_A^1 \rangle_{\mathscr{H}}.
\end{align}
We use \cite[(1.48)]{Nua06} again to write
\begin{align}\label{eq2017-08-28-7}
\delta(u_A^2/\gamma_A^{2, 2}) &= \delta(u_A^2)/\gamma_A^{2,2} + \langle D\gamma_A^{2, 2}, u_A^2 \rangle_{\mathscr{H}}/(\gamma_A^{2, 2})^2.
\end{align}
Hence the first term on the right-hand side of (\ref{eq2017-08-28-5}) is equal to
\begin{align}\label{eq2017-08-28-6}
\delta(u_A^2/\gamma_A^{2, 2})\delta(u_A^1) &= \frac{\delta(u_A^2)}{\gamma_A^{2, 2}}\delta(u_A^1) + \frac{\langle D\gamma_A^{2, 2}, u_A^2 \rangle_{\mathscr{H}}}{(\gamma_A^{2,2})^2}\delta(u_A^1).
\end{align}
For the second term on the right-hand side of (\ref{eq2017-08-28-5}), we apply \eqref{eq2017-08-28-7} to obtain that
\begin{align}\label{eq2017-08-28-8}
D\delta(u_A^2/\gamma_A^{2, 2}) &= D(\delta(u_A^2)/\gamma_A^{2,2}) - D(\langle D\gamma_A^{2, 2}, u_A^2 \rangle_{\mathscr{H}}/(\gamma_A^{2, 2})^2)\nonumber \\
&= \frac{D\delta(u_A^2)}{\gamma_A^{2,2}} - \frac{\delta(u_A^2)D\gamma_A^{2,2}}{(\gamma_A^{2,2})^2} - \frac{D\langle D\gamma_A^{2, 2}, u_A^2 \rangle_{\mathscr{H}}}{(\gamma_A^{2, 2})^2} + \frac{2\langle D\gamma_A^{2, 2}, u_A^2 \rangle_{\mathscr{H}}D\gamma_A^{2, 2}}{(\gamma_A^{2, 2})^3}.
\end{align}
Therefore the second term on the right-hand side of (\ref{eq2017-08-28-5}) can be written as
\begin{align}\label{eq2017-08-28-9}
- \langle D\delta(u_A^2/\gamma_A^{2, 2}), u_A^1 \rangle_{\mathscr{H}} &= - \frac{1}{\gamma_A^{2, 2}}\langle D\delta(u_A^2), u_A^1 \rangle_{\mathscr{H}} - \frac{2\langle D\gamma_A^{2, 2}, u_A^2 \rangle_{\mathscr{H}}}{(\gamma_A^{2,2})^3}\langle D\gamma_A^{2, 2}, u_A^1 \rangle_{\mathscr{H}} \nonumber\\
  &  \quad + \frac{\delta(u_A^2)}{(\gamma_A^{2, 2})^2} \langle D\gamma_A^{2, 2}, u_A^1 \rangle_{\mathscr{H}}  + \frac{1}{(\gamma_A^{2, 2})^2}\langle D\langle D\gamma_A^{2, 2}, u_A^2 \rangle_{\mathscr{H}}, u_A^1 \rangle_{\mathscr{H}}.
\end{align}
Putting (\ref{eq2017-08-28-6}) and (\ref{eq2017-08-28-9}) together, we obtain \eqref{eq2017-08-28-10}.
\end{proof}

\begin{prop}\label{prop2017-11-27-1}
\begin{itemize}
  \item [(a)]For any $p \geq 2$, there exists $c_p > 0$, not depending on $(s_0, y_0) \in I \times J$, such that for all small $\delta_1 > 0$,  and for all $z_2 \geq \delta_1^{1/4}$,
      \begin{align} \label{eq2017-12-19-4}
       \|T_i\|_{L^p(\Omega)} \leq c_p\, \delta_1^{-1/4}, \quad \mbox{for} \,\, \, i \in \{1, 2, 3\}.
       \end{align}
  \item [(b)] $T_4, T_5$ and $T_6$ vanish.
  \end{itemize}
\end{prop}

An immediate consequence of Lemma \ref{lemma2017-12-17-1} and Proposition \ref{prop2017-11-27-1} is the following.
\begin{prop}\label{prop0511-1}
There exists a finite positive constant $c$, not depending on $(s_0, y_0) \in I \times J$, such that for all small $\delta_1 > 0$ and for all $z_2 \geq \delta_1^{1/4}$,
\begin{align} \label{eq2017-05-11-10}
\|\delta(\delta(u_A^2/\gamma_A^{2, 2})u_A^1)\|_{L^2(\Omega)} \leq c \, \delta_1^{-1/4}.
\end{align}
\end{prop}

The proof of Proposition \ref{prop2017-11-27-1} is divided into the following two subsections.

\subsection{Proof of Proposition \ref{prop2017-11-27-1}(a)} \label{section4.5.1}

Throughout  Section \ref{section4.5.1}, we assume that
\begin{align}\label{eq2017-05-04-6}
z_2 \geq \delta_1^{1/4}.
\end{align}
Recalling the definition of $R$ in \eqref{eq2017-10-12-5}, under the assumption \eqref{eq2017-05-04-6}, we see from \eqref{eq2017-09-28-4} that
\begin{align}\label{eq2017-12-1-3}
R^{-1} &= c^{-1}a^{-2p_0}\delta_1^{(\gamma_0 - 4)/2} = c'z_2^{-2p_0}\delta_1^{(\gamma_0 - 4)/2}  \leq c'\delta_1^{(\gamma_0 - p_0 - 4)/2}.
\end{align}
We will make use of this in the estimates below.

We first give an estimate for the moments of $T_1$. In order to estimate the moments of the Skorohod integral $\delta(u_A^2)$, recall the extension of Proposition 1.3.11 of \cite{Nua06} to multiparameter adapted processes, mentioned in \cite[p.45]{Nua06}.

We denote by $L^2_a$ the closed subspace of $L^2(\Omega \times [0, T] \times [0, 1])$ formed by those processes which are adapted to the filtration $(\mathscr{F}_t)_{t\geq 0}$ defined in \eqref{eq2018-07-31-1}.

\begin{prop} \label{lemma2017-05-04-2}
$L^2_a \subset \mbox{Dom} \ \delta$ and the operator $\delta$ restricted to $L^2_a$ coincides with the Walsh integral, that is, for $u \in L^2_a$,
\begin{align}\label{eq2017-05-05-1}
\delta(u) = \int_0^T\int_0^1 u(r, v)W(dr, dv).
\end{align}
\end{prop}
\noindent{}The proof is similar to that of \cite[Proposition 1.3.11]{Nua06} and is omitted.

Proposition \ref{lemma2017-05-04-2} enables us to use properties of Walsh integrals to estimate the $L^p(\Omega)$-norm of $\delta(u_A^2)$, as in the following lemma.

\begin{lemma} \label{lemma05-04-1}
For any $p \geq 2$, there exists a constant $c_p$, not depending on $(s_0, y_0) \in I \times J$, such that for all $\delta_1 > 0$,
\begin{align} \label{eq2017-05-04-1}
\|\delta(u_A^2)\|_{L^p(\Omega)} \leq c_p\delta_1^{3/4}.
\end{align}
\end{lemma}

\begin{proof}
From  \eqref{eq2017-11-26-8}, we know that for $(r, v) \in \, ]s_0, s_0 + \delta_1] \times [0, 1]$,
\begin{align*}
u_A^2(r, v) = \phi_{\delta_1}(v)\psi(Y_r) - \phi_{\delta_1}''(v)\int_{s_0}^r\psi(Y_a)da.
\end{align*}
Since $u_A^2$ is adapted, by Proposition \ref{lemma2017-05-04-2}, we have
\begin{align} \label{eq2017-05-04-3}
\delta(u_A^2) &= \int_{s_0}^{s_0 + \delta_1}\int_0^1\phi_{\delta_1}(v)\psi(Y_r)W(dr, dv) - \int_{s_0}^{s_0 + \delta_1}\int_0^1W(dr, dv)\phi_{\delta_1}''(v)\int_{s_0}^r\psi(Y_a)da.
\end{align}
For the first term on the right-hand side of \eqref{eq2017-05-04-3}, by Burkholder's inequality, for any $p \geq 2$,  since $0 \leq \psi \leq 1$,
\begin{align}\label{eq2017-05-04-4}
& \left|\left|\int_{s_0}^{s_0 + \delta_1}\int_0^1\phi_{\delta_1}(v)\psi(Y_r)W(dr, dv)\right|\right|_{L^p(\Omega)}^p \nonumber \\
& \quad \leq c_p\mbox{E}\left[\left(\int_{s_0}^{s_0 + \delta_1}\int_0^1\phi_{\delta_1}^2(v)\psi^2(Y_r)drdv\right)^{p/2}\right]\nonumber\\
 & \quad \leq c_p \delta_1^{p/2}\left(\int_0^1\phi_{\delta_1}^2(v)dv\right)^{p/2} \leq c'_p \delta_1^{p/2}\delta_1^{p/4} =  c'_p \delta_1^{3p/4}.
\end{align}
For the second term on the right-hand side of \eqref{eq2017-05-04-3}, similarly, by Burkholder's inequality, for any $p \geq 2$, since $0 \leq \psi \leq 1$,
\begin{align}\label{eq2017-05-04-5}
& \left|\left|\int_{s_0}^{s_0 + \delta_1}\int_0^1W(dr, dv)\phi_{\delta_1}''(v)\int_{s_0}^r\psi(Y_a)da\right|\right|_{L^p(\Omega)}^p \nonumber \\
  & \quad \leq c_p\mbox{E}\left[\left(\int_{s_0}^{s_0 + \delta_1}dr\int_0^1dv(\phi_{\delta_1}''(v))^2\left(\int_{s_0}^r\psi(Y_a)da\right)^2\right)^{p/2}\right]\nonumber\\
& \quad \leq c_P \left(\int_{s_0}^{s_0 + \delta_1}(r - s_0)^2dr\right)^{p/2}\left(\int_0^1(\phi_{\delta_1}''(v))^2dv\right)^{p/2}\nonumber\\
& \quad \leq c_p \delta_1^{3p/2}\left(\int_{y_0 - \delta_1^{1/2}}^{y_0 + 2\delta_1^{1/2}}\delta_1^{-2}dv\right)^{p/2}  =  c'_p \delta_1^{3p/2}\delta_1^{-3p/4} =  c'_p \delta_1^{3p/4},
\end{align}
where, in the third inequality, we use \eqref{eq2017-11-26-4}.
Hence \eqref{eq2017-05-04-1} follows from \eqref{eq2017-05-04-3}, \eqref{eq2017-05-04-4} and \eqref{eq2017-05-04-5}.
\end{proof}

Since $u_A^1$ is deterministic, by \eqref{eq2017-12-26-4}, for any $p \geq 1$,
\begin{align}\label{eq2017-8}
\|\delta(u_A^1)\|_{L^p(\Omega)} = c_p\left(\int_0^T\int_0^1(u_A^1(r, v))^2drdv\right)^{1/2} \leq c'_p .
\end{align}

From (\ref{eq2017-6}), (\ref{eq2017-05-04-1}) and (\ref{eq2017-8}), using H\"{o}lder's inequality, we obtain that for all $p \geq 2$
\begin{align}\label{eq2017-9}
\|T_1\|_{L^p(\Omega)} \leq c_p\delta_1^{-1}\delta_1^{3/4}= c_p\delta_1^{-1/4}.
\end{align}
This proves the statement (a) of Proposition \ref{prop2017-11-27-1} for $i = 1$.

Next, we show that the estimate in Proposition \ref{prop2017-11-27-1}(a) holds for $T_2$.
We first use the formula \eqref{eq2017-11-26-8} to give an estimate on the $\mathscr{H}$-norm of $u_A^2$. By definition, since $0 \leq \psi \leq 1$,
\begin{align}\label{eq2017-13}
\|u_A^2\|_{\mathscr{H}}^2 &\leq 2\int_{s_0}^{s_0 + \delta_1}dr\int_0^1dv \, \psi(Y_r)^2\phi_{\delta_1}^2(v) + 2\int_{s_0}^{s_0 + \delta_1}dr\int_0^1dv\, (\phi_{\delta_1}''(v))^2\left(\int_{s_0}^r\psi(Y_a)da\right)^2 \nonumber\\
&\leq 2\delta_1 \int_{y_0 - \delta_1^{1/2}}^{y_0 + 2\delta_1^{1/2}}dv  + 2c\int_{s_0}^{s_0 + \delta_1}(r - s_0)^2dr\int_{y_0 - \delta_1^{1/2}}^{y_0 + 2\delta_1^{1/2}}\delta_1^{-2}dv \nonumber\\
&= c \, \delta_1^{3/2} + c \, \delta_1^{3}\delta_1^{-3/2}  = 2c \, \delta_1^{3/2},
\end{align}
where, in the second inequality, we use \eqref{eq2017-11-26-4}.
\begin{lemma}\label{lemma2017-08-29-1}
For any $p \geq 1$, there exists a constant $c_p$, not depending on $(s_0, y_0) \in I \times J$, such that for all $\delta_1 > 0$,
\begin{align}\label{eq2017-14}
\|\langle D\gamma_A^{2, 2}, u_A^2 \rangle_{\mathscr{H}}\|_{L^p(\Omega)} \leq c_p \delta_1^{7/4}.
\end{align}
\end{lemma}
\begin{proof}
Taking the Malliavin derivative of $\gamma_A^{2, 2}$, we have
\begin{align*}
\langle D\gamma_A^{2, 2}, u_A^2 \rangle_{\mathscr{H}} = \int_{s_0}^{s_0 + \delta_1}\psi'(Y_r)\langle DY_r, u_A^2 \rangle_{\mathscr{H}}dr.
\end{align*}
By H\"{o}lder's inequality,  \eqref{eq2017-11-26-1} and \eqref{eq2017-13}, for any $p \geq 1$,
\begin{align*}
\mbox{E}\left[|\langle D\gamma_A^{2, 2}, u_A^2 \rangle_{\mathscr{H}}|^p\right] &\leq \|\psi'\|^p_{\infty}\delta_1^{p - 1}\int_{s_0}^{s_0 + \delta_1}\mbox{E}[|\langle DY_r, u_A^2 \rangle_{\mathscr{H}}|^p]dr \nonumber \\
&\leq c_pR^{-p}\delta_1^{p -1}\int_{s_0}^{s_0 + \delta_1}\mbox{E}[\|DY_r\|^p_{\mathscr{H}}\|u_A^2\|_{\mathscr{H}}^p]dr \nonumber \\
&\leq  c_pR^{-p}\delta_1^{p - 1 + 3p/4}\int_{s_0}^{s_0 + \delta_1}\mbox{E}[\|DY_r\|^p_{\mathscr{H}}]dr. \nonumber
\end{align*}
Using Lemma \ref{lemma2017-10-2-1}(a), this is bounded above by
\begin{align*}
& c_pR^{-p}\delta_1^{p - 1 + 3p/4}\delta_1^{(p_0 - \gamma_0)p/2}\int_{s_0}^{s_0 + \delta_1}(r - s_0)^{2p}dr \nonumber \\
&\quad = c_pR^{-p}\delta_1^{p - 1 + 3p/4}\delta_1^{(p_0 - \gamma_0)p/2}\delta_1^{2p + 1}\\
& \quad \leq c_p \delta_1^{(\gamma_0 - p_0 - 4)p/2}\delta_1^{p - 1 + 3p/4}\delta_1^{(p_0 - \gamma_0)p/2}\delta_1^{2p + 1}= c_p \delta_1^{7p/4},
\end{align*}
where, in the inequality, we use \eqref{eq2017-12-1-3}.
\end{proof}
By (\ref{eq2017-6}), (\ref{eq2017-8}) and (\ref{eq2017-14}), using H\"{o}lder's inequality, we obtain that for any $p \geq 1$
\begin{align}\label{eq2017-15}
\|T_2\|_{L^p(\Omega)} \leq c_p \delta_1^{-2}\delta_1^{7/4} = c_p\delta_1^{-1/4}.
\end{align}
This proves the statement (a) of Proposition \ref{prop2017-11-27-1} for $i = 2$.

We proceed to give an estimate on the moments of $T_3$.
Using \eqref{eq2017-05-04-3}, we take the Malliavin derivative of $\delta(u_A^2)$ and write
\begin{align}\label{eq2017-08-29-1}
D_{\xi, \eta}\delta(u_A^2) &= 1_{[s_0, s_0 + \delta_1]}(\xi)\psi(Y_{\xi})\phi_{\delta_1}(\eta) - 1_{[s_0, s_0 + \delta_1]}(\xi)\phi_{\delta_1}''(\eta)\int_{s_0}^{\xi}\psi(Y_a)da\nonumber \\
&\quad  + \int_{s_0}^{s_0 + \delta_1}\int_0^1\phi_{\delta_1}(v)\psi'(Y_r)D_{\xi, \eta}Y_r \, W(dr, dv) \nonumber\\
&\quad - \int_{s_0}^{s_0 + \delta_1}\int_0^1W(dr, dv)\phi_{\delta_1}''(v)\int_{s_0}^{r}\psi'(Y_a)D_{\xi, \eta}Y_a \, da.
\end{align}
It is clear that the inner product of the first two terms on the right-hand side of (\ref{eq2017-08-29-1}) and $u_A^1$ is equal to $\langle u_A^2, u_A^1 \rangle_{\mathscr{H}}$.
By the stochastic Fubini theorem (see  \cite[Chapter 1, Theorem 5.30]{DKMNXR08} or  \cite[Theorem 2.6]{Wal86}), we see that the inner product of the third term on the right-hand side of (\ref{eq2017-08-29-1}) and $u_A^1$ is equal to
\begin{align}\label{eq2017-08-29-2}
 \int_{s_0}^{s_0 + \delta_1}\int_0^1\phi_{\delta_1}(v)\psi'(Y_r)\langle DY_r, u_A^1 \rangle_{\mathscr{H}}W(dr, dv),
\end{align}
since the condition of the stochastic Fubini theorem can be verified:
\begin{align*}
& \mbox{E}\left[\int_{s_0}^{s_0 + \delta_1}d\xi \int_0^1d\eta \, |u_A^1(\xi, \eta)|\int_{s_0}^{s_0 + \delta_1}dr\int_0^1dv \, \phi_{\delta_1}^2(v)(\psi'(Y_r))^2(D_{\xi, \eta}Y_r)^2 \right] \\
&  \quad \quad \leq c\mbox{E}\left[\int_{s_0}^{s_0 + \delta_1}d\xi \int_0^1d\eta \int_{s_0}^{s_0 + \delta_1}(D_{\xi, \eta}Y_r)^2dr \right] \\
&  \quad \quad \leq c \delta_1\sup_{r \in [s_0, s_0 + \delta_1]}\mbox{E}\left[\|DY_{r}\|_{\mathscr{H}}^2\right] < \infty,
\end{align*}
where the last inequality is due to (\ref{eq2017-11-26-9}). Similarly, the inner product of the last term on the right-hand side of (\ref{eq2017-08-29-1}) and $u_A^1$ is equal to
\begin{align}\label{eq2017-08-29-3}
\int_{s_0}^{s_0 + \delta_1}\int_0^1W(dr, dv)\phi_{\delta_1}''(v)\int_{s_0}^{r}\psi'(Y_a)\langle DY_a, u_A^1 \rangle_{\mathscr{H}}da.
\end{align}
Therefore, by (\ref{eq2017-08-29-1}), (\ref{eq2017-08-29-2}) and (\ref{eq2017-08-29-3}),  we write
\begin{align}\label{eq2017-05-11-1}
\langle D\delta(u_A^2), u_A^1 \rangle_{\mathscr{H}} &=  \langle u_A^2, u_A^1 \rangle_{\mathscr{H}} + \int_{s_0}^{s_0 + \delta_1}\int_0^1\psi'(Y_r)\langle DY_r, u_A^1 \rangle_{\mathscr{H}}\phi_{\delta_1}(v)W(dr, dv)\nonumber \\
& \quad - \int_{s_0}^{s_0 + \delta_1}\int_0^1W(dr, dv) \, \phi''_{\delta_1, \delta_2}(v)\int_{s_0}^rda \, \psi'(Y_a)\langle DY_a, u_A^1 \rangle_{\mathscr{H}}\nonumber \\
&:= \bar{T}_{31} + \bar{T}_{32} - \bar{T}_{33}.
\end{align}
From (\ref{eq2017-13}) and \eqref{eq2017-12-26-4}, it is clear that for any $p \geq 1$,
\begin{align}\label{eq2017-05-11-2}
\|\bar{T}_{31}\|_{L^p(\Omega)} \leq c_p \delta_1^{3/4}.
\end{align}
By Burkholder's inequality and using \eqref{eq2017-11-26-1} and \eqref{eq2017-12-26-4}, we have for any $p \geq 2$,
\begin{align}\label{eq2017-05-11-3}
\mbox{E}[|\bar{T}_{32}|^p] &\leq c_p \mbox{E}\left[\left(\int_{s_0}^{s_0 + \delta_1}\int_0^1\psi'(Y_r)^2\langle DY_r, u_A^1 \rangle_{\mathscr{H}}^2\phi_{\delta_1}^2(v)drdv\right)^{p/2}\right] \nonumber \\
&\leq c_p R^{-p}\mbox{E}\left[\left(\int_{s_0}^{s_0 + \delta_1}\|DY_r\|_{\mathscr{H}}^2dr\int_0^1\phi_{\delta_1}^2(v)dv\right)^{p/2}\right]\nonumber \\
&= c_p R^{-p}\left(\int_0^1\phi_{\delta_1}^2(v)dv\right)^{p/2}\mbox{E}\left[\left(\int_{s_0}^{s_0 + \delta_1}\|DY_r\|_{\mathscr{H}}^2dr\right)^{p/2}\right].
\end{align}
By H\"{o}lder's inequality and (\ref{eq2017-10}), we see that (\ref{eq2017-05-11-3}) is bounded above by
\begin{align}\label{eq2017-08-31-1}
&  c_p R^{-p}\delta_1^{p/4}\delta_1^{p/2 - 1}\int_{s_0}^{s_0 + \delta_1}\mbox{E}[\|DY_r\|_{\mathscr{H}}^p]dr\nonumber \\
&  \quad \leq c_p R^{-p}\delta_1^{p/4}\delta_1^{p/2 - 1}\delta_1^{(p_0 - \gamma_0)p/2}\int_{s_0}^{s_0 + \delta_1}(r - s_0)^{2p}dr\nonumber \\
&  \quad = c_p R^{-p} \delta_1^{(2(p_0 - \gamma_0) + 11)p/4} \leq c'_p \delta_1^{(\gamma_0 - p_0 - 4)p/2}\delta_1^{(2(p_0 - \gamma_0) + 11)p/4} = c'_p \delta_1^{3p/4},
\end{align}
where in the last inequality we use \eqref{eq2017-12-1-3}.

We now give an estimate on the moments of $\bar{T}_{33}$. By Burkholder's inequality and using \eqref{eq2017-11-26-1} and \eqref{eq2017-12-26-4}, for any $p \geq 2$,
\begin{align}\label{eq2017-05-11-4}
\mbox{E}[|\bar{T}_{33}|^p] &\leq c_p\mbox{E}\left[\left(\int_{s_0}^{s_0 + \delta_1}\int_0^1\left(\int_{s_0}^r\psi'(Y_a)\langle DY_a, u_A^1 \rangle_{\mathscr{H}}da\right)^2(\phi_{\delta_1}''(v))^2drdv\right)^{p/2}\right] \nonumber \\
&\leq c_pR^{-p}\left(\int_0^1(\phi_{\delta_1}''(v))^2dv\right)^{p/2}\mbox{E}\left[\left(\int_{s_0}^{s_0 + \delta_1}\left(\int_{s_0}^r\|DY_a\|_{\mathscr{H}}da\right)^2dr\right)^{p/2}\right].
\end{align}
Using H\"{o}lder's inequality twice and \eqref{eq2017-11-26-4}, \eqref{eq2017-05-11-4} is bounded above by
\begin{align}\label{eq2017-08-31-2}
&c_pR^{-p}\delta_1^{-3p/4}\mbox{E}\left[\left(\int_{s_0}^{s_0 + \delta_1}dr(r - s_0)\int_{s_0}^r\|DY_a\|^2_{\mathscr{H}}da\right)^{p/2}\right] \nonumber \\
&\quad \leq c_pR^{-p}\delta_1^{-3p/4}\left(\int_{s_0}^{s_0 + \delta_1}dr\int_{s_0}^rda\right)^{p/2 - 1}\int_{s_0}^{s_0 + \delta_1}dr(r - s_0)^{p/2}\int_{s_0}^r\mbox{E}[\|DY_a\|^p_{\mathscr{H}}]da.
\end{align}
Applying the estimate in (\ref{eq2017-10}), (\ref{eq2017-08-31-2}) is bounded above by
\begin{align}\label{eq2017-08-31-3}
& c_pR^{-p}\delta_1^{-3p/4}\delta_1^{p -2}\delta_1^{(p_0 - \gamma_0)p/2}\int_{s_0}^{s_0 + \delta_1}dr(r - s_0)^{p/2}\int_{s_0}^r(a - s_0)^{2p}da \nonumber \\
&  \quad =c_pR^{-p}\delta_1^{(2(p_0 - \gamma_0) + 11)p/4} \leq c'_p \delta_1^{(\gamma_0 - p_0 - 4)p/2}\delta_1^{(2(p_0 - \gamma_0) + 11)p/4} = c'_p \delta_1^{3p/4},
\end{align}
where in the inequality we use \eqref{eq2017-12-1-3}.

Therefore, by (\ref{eq2017-05-11-2}),(\ref{eq2017-08-31-1}), (\ref{eq2017-08-31-3}) and (\ref{eq2017-6}), we have obtained that for any $p \geq 2$,
\begin{align}\label{eq2017-05-11-5}
\|T_3\|_{L^p(\Omega)} \leq c_p\delta_1^{-1/4}.
\end{align}
This proves the statement (a) of Proposition \ref{prop2017-11-27-1} for $i = 3$.

Therefore, we have finished the proof of Proposition \ref{prop2017-11-27-1}(a).

\subsection{Proof of Proposition \ref{prop2017-11-27-1}(b)} \label{section4.5.2}

We are going to show that the three terms $T_4$, $T_5$ and $T_6$ are equal to zero. First, we apply Lemma \ref{lemma2017-11-27-1} to see that
for any $t, s \in [s_0, s_0 + \delta_1]$,
\begin{align} \label{eq2017-08-31-4}
\langle D(u(t, y_0) - u(s, y_0)), u_A^1 \rangle_{\mathscr{H}} &= \int_0^T\int_{0}^1 (1_{\{r < t\}}G_{\alpha}(t - r, y_0, v) - 1_{\{r < s\}}G_{\alpha}(s - r, y_0, v))\nonumber\\
& \quad \quad \quad \quad  \times \left(\frac{\partial}{\partial r} - \frac{\partial^2}{\partial v^2}\right)(f_0(r)g_0(v))drdv\nonumber \\
&= f_0(t)g_0(x) - f_0(s)g_0(y_0) = 1 - 1 = 0,
\end{align}
by the definition of the functions $f_0$ and $g_0$.
Furthermore, by (\ref{eq2017-11}) and (\ref{eq2017-08-31-4}), we know that for $r \in [s_0, s_0 + \delta_1]$,
\begin{align}\label{eq2017-05-11-6}
\langle DY_r, u_A^1 \rangle_{\mathscr{H}} &= 2p_0\int_{[s_0, r]^2}dsdt\, \frac{(u(t, y_0) - u(s, y_0))^{2p_0 - 1}}{|t - s|^{\gamma_0/2}}\langle D(u(t, y_0) - u(s, y_0)), u_A^1 \rangle_{\mathscr{H}}\nonumber \\
&= 0.
\end{align}
Hence, by \eqref{eq2017-11-28-1},
\begin{align}\label{eq2017-05-11-7}
\langle D\gamma_A^{2, 2}, u_A^1 \rangle_{\mathscr{H}} = \int_{s_0}^{y_0 + \delta_1} \psi'(Y_r)\langle DY_r, u_A^1 \rangle_{\mathscr{H}}dr = 0,
\end{align}
which implies that $T_4 = T_5 = 0$.

We proceed to prove that $T_6$ vanishes. Similar to  \eqref{eq2017-08-31-4}, for any $t, s \in [s_0, s_0 + \delta_1]$,
\begin{align} \label{eq2017-10-03-1}
& \langle D(u(t, y_0) - u(s, y_0)), u_A^2 \rangle_{\mathscr{H}} \nonumber\\
& \quad = \int_{s_0}^{s_0 + \delta_1}dr\int_0^1dv(1_{\{r < t\}}G(t - r, y_0, v) - 1_{\{r < s\}}G(s - r, y_0, v))\left(\frac{\partial}{\partial r} - \frac{\partial^2}{\partial v^2}\right)H(r, v)\nonumber \\
& \quad = H(t, y_0) - H(s, y_0).
\end{align}
Hence, by \eqref{eq2017-09-28-3}, for $r \in [s_0, s_0 + \delta_1]$,
\begin{align} \label{eq2017-05-11-8}
\langle DY_r, u_A^2 \rangle_{\mathscr{H}} &= 2p_0\int_{[s_0, r]^2}dsdt\,\frac{(u(t, y_0) - u(s, y_0))^{2p_0 - 1}}{|t - s|^{\gamma_0/2}}\langle D(u(t, y_0) - u(s, y_0)), u_A^2 \rangle_{\mathscr{H}}\nonumber  \\
 &= 2p_0\int_{[s_0, r]^2}dsdt\,\frac{(u(t, y_0) - u(s, y_0))^{2p_0 - 1}}{|t - s|^{\gamma_0/2}}(H(t, y_0) - H(s, y_0)) \nonumber\\
&= 2p_0\int_{[s_0, r]^2}dsdt \, \frac{(u(t, y_0) - u(s, y_0))^{2p_0 - 1}}{|t - s|^{\gamma_0/2}}\int_s^t\psi(Y_a)da,
\end{align}
where in the last equality we use the definition of the function $(t, x) \mapsto H(t, x)$. Moreover,
\begin{align} \label{eq2017-05-11-100}
& \langle D\langle DY_r, u_A^2 \rangle_{\mathscr{H}}, u_A^1\rangle_{\mathscr{H}} \nonumber \\
& \quad = 2p_0(2p_0 - 1)\int_{[s_0, r]^2}dsdt\, \frac{(u(t, y_0) - u(s, y_0))^{2p_0 - 2}}{|t - s|^{\gamma_0/2}} \nonumber \\
&  \quad \quad\quad \qquad\quad\quad \quad\quad \quad \times \langle D(u(t, y_0) - u(s, y_0)), u_A^1 \rangle_{\mathscr{H}}\int_s^t\psi(Y_a)da \nonumber \\
& \quad  \quad   + 2p_0\int_{[s_0, r]^2}dsdt\, \frac{(u(t, y_0) - u(s, y_0))^{2p_0 - 1}}{|t - s|^{\gamma_0/2}}\int_s^t\psi'(Y_a)\langle DY_a, u_A^1 \rangle_{\mathscr{H}}da \nonumber \\
& \quad = 0 + 0 = 0,
\end{align}
where, on the right-hand side of the equality, the first term vanishes due to (\ref{eq2017-08-31-4}) and the second term vanishes because of (\ref{eq2017-05-11-6}). Therefore, by definition of $\gamma_A^{2, 2}$,
\begin{align} \label{eq2017-05-11-9}
\langle D\langle D\gamma_A^{2, 2}, u_A^2 \rangle_{\mathscr{H}}, u_A^1\rangle_{\mathscr{H}} &= \left\langle D \int_{s_0}^{s_0 + \delta_1}\psi'(Y_r)\langle DY_r, u_A^2 \rangle_{\mathscr{H}}dr, u_A^1\right\rangle_{\mathscr{H}} \nonumber \\
  &=  \int_{s_0}^{s_0 + \delta_1}\psi''(Y_r)\langle DY_r, u_A^1 \rangle_{\mathscr{H}}\langle DY_r, u_A^2 \rangle_{\mathscr{H}}dr\nonumber \\
  & \quad  + \int_{s_0}^{s_0 + \delta_1}\psi'(Y_r) \langle D\langle DY_r, u_A^2 \rangle_{\mathscr{H}}, u_A^1\rangle_{\mathscr{H}}dr \nonumber \\
   &= 0,
\end{align}
by \eqref{eq2017-05-11-6} and \eqref{eq2017-05-11-100}, which implies $T_6 = 0$.

This proves the statement (b) of Proposition \ref{prop2017-11-27-1}.

\subsection{Estimates for the tail probabilities} \label{section4.4.5}

\begin{lemma}
There exists a finite positive constant $c$, not depending on $(s_0, y_0) \in I \times J$, such that for all  $z_1 \in \mathbb{R}$,
\begin{align}\label{eq2017-09-07-1400}
 \mbox{P}\{|F_1| > |z_1|\} \leq c\, (|z_1|^{-1}\wedge 1)e^{-z_1^2/c},
\end{align}
and for all $\delta_1 > 0$ and $z_2 > 0$,
\begin{align}\label{eq2016-12-27-20000}
\mbox{P}\{F_2 > z_2\} \leq c\, \exp(-z_2^2/(c\, \delta_1^{1/2})).
\end{align}
\end{lemma}
\begin{proof}
We first bound $\mbox{P}\{|F_1| > |z_1|\}$. Since the variance of $u(s_0, y_0)$ is bounded above and below by positive constants uniformly over $(s_0, y_0) \in I \times J$  (see \cite[(4.5)]{DKN07}), there are constants $c_1, c_2, c_3, c_4$ independent of $(s_0, y_0) \in I \times J$ such that for all $z_1 \in \mathbb{R}$
\begin{align}\label{eq2017-09-07-14}
 \mbox{P}\{|F_1| > |z_1|\} \leq c_1\int_{|z_1|}^{+\infty}e^{-y^2/c_2}dy \leq c_3(|z_1|^{-1}\wedge 1)e^{-z_1^2/c_4},
\end{align}
where the last inequality holds by \cite[7.1.13, p.298]{AlS79}. This proves \eqref{eq2017-09-07-1400}.

We denote
$
\sigma^2 := \sup_{t \in [s_0, s_0 + \delta_1]}\mbox{E}[\bar{u}(t, y_0)^2].
$
 From \eqref{eq2016-06-14-1}, we have $\sigma^2 \leq C \, \delta_1^{1/2}$. On the other hand, by \cite[(4.50)]{DKN07}, we have
\begin{align}\label{eq2017-09-15}
\mbox{E}[F_2] &\leq \mbox{E}\left[\sup\limits_{t \in [s_0, s_0 + \delta_1] }|u(t, y_0) - u(s_0, y_0)|\right] \nonumber\\
&\leq \mbox{E}\left[\sup\limits_{[|t - s_0|^{1/2} + |x - y_0|]^{1/2}\leq \delta_1^{1/4}}|u(t, x) - u(s_0, y_0)|\right] \leq c \, \delta_1^{1/4}.
\end{align}
Applying Borell's inequality (see \cite[(2.6)]{Adl90}) and the fact that $(z_2 - \mathrm{E}[F_2])^2 \geq \frac{2}{3}z_2^2 - 2\mathrm{E}[F_2]^2$, we see that for all $z_2 > c\,  \delta_1^{1/4}$ (here  $c$ is the constant in \eqref{eq2017-09-15}),
\begin{align} \label{eq2018-12-11-1}
\mbox{P}\{F_2 > z_2\} &\leq 2 \exp\left(-(z_2 - \mbox{E}[F_2])^2/(2\sigma^2)\right) \leq \bar{c}\exp\left(-z_2^2/(3C\delta_1^{1/2})\right).
\end{align}
Since for $0 \leq z_2 \leq c \delta_1^{1/4}$,
$
\exp(-z_2^2/(3C\delta_1^{1/2})) \geq e^{-c^2/(3C)},
$
we can find a constant $\tilde{c}$ such that for all $z_2 > 0$,
\begin{align}\label{eq2016-12-27-2}
\mbox{P}\{F_2 > z_2\} \leq \tilde{c}\exp(-z_2^2/(3C\delta_1^{1/2})).
\end{align}
This proves \eqref{eq2016-12-27-20000}.
\end{proof}

Finally, we prove Theorem \ref{theorem100}.

\begin{proof}[Proof of Theorem \ref{theorem100}]
This follows from (\ref{eq2017-09-07-13}), \eqref{eq2017-09-07-1400}, \eqref{eq2016-12-27-20000} and \eqref{eq2017-05-11-10}.
\end{proof}

\section{Gaussian-type upper bound on the density of $M_0$}\label{section4.7}
In this section, we assume $J \subset \, ]0, 1[$ and $\delta_1$, $\delta_2$ satisfy the conditions in \eqref{eq2017-09-2600}.

From the formula for the probability density function of $M_0$ in \eqref{eq2018-01-02-15}, by the Cauchy-Schwartz inequality,
\begin{align} \label{eq2018-01-04-4}
p_0(z) \leq \mbox{P}\{M_0 > z \}^{1/2}\|\delta(u_{\bar{A}}/\gamma_{\bar{A}})\|_{L^2(\Omega)}.
\end{align}

\begin{prop}\label{prop0511-100}
\begin{itemize}
  \item [(a)]There exists a finite positive constant $c$, not depending on $y_0 \in  J$, such that for all small $\delta_1, \, \delta_2 > 0$ and for all $z \geq (\delta_1^{1/2} + \delta_2)^{1/2}$,
 \begin{align}\label{eq2018-01-02-17}
                     \|\delta(u_{\bar{A}}/\gamma_{\bar{A}})\|_{L^2(\Omega)} \leq c \, (\delta_1^{1/2} + \delta_2)^{-1/2}.
                     \end{align}
  \item [(b)] There exists a finite positive constant $c$, not depending on $y_0 \in  J$, such that for all $\delta_1, \, \delta_2 > 0$ and for all $z > 0$,
  \begin{align}\label{eq2018-01-4-3}
  \mbox{P}\{M_0 > z \} \leq c \, \exp(- z^2/(c \, (\delta_1^{1/2} + \delta_2))).
  \end{align}
\end{itemize}
\end{prop}

\begin{proof}[Proof of Theorem \ref{theorem2018-01-03-1}]
This is an immediate  consequence of \eqref{eq2018-01-04-4} and Proposition \ref{prop0511-100}.
\end{proof}
The proof of Proposition \ref{prop0511-100} is given in the following two subsections.

\subsection{Proof of Proposition \ref{prop0511-100}(a)}

Throughout this section, we assume that
\begin{align}\label{eq2017-05-04-60011}
z \geq (\delta_1^{1/2} + \delta_2)^{1/2} = \delta^{1/2}.
\end{align}
Recalling the definition of $\bar{R}$ in \eqref{eq2017-10-12-500}, under the assumption \eqref{eq2017-05-04-60011}, we see from \eqref{eq2017-09-28-411} that
\begin{align}\label{eq2017-12-1-30011}
\bar{R}^{-1} &= c^{-1}\bar{a}^{-2p_0}\delta^{\gamma_0 - 4} = c'z^{-2p_0}\delta^{\gamma_0 - 4}  \leq c' \, \delta^{\gamma_0 - p_0 - 4}.
\end{align}

In order to prove Proposition \ref{prop0511-100}(a), we need the following lemmas. Recall the definition of $\Delta_{\bullet}$ in \eqref{eq2018-01-03-1} and of $u_{\bar{A}}$ in \eqref{eq2017-11-26-511}; and notice that for $(r, v) \in [0, \Delta_{\bullet}] \times [0, 1]$,
\begin{align}\label{eq2018-10-09-2}
u_{\bar{A}}(r, v) = \bar{\phi}_{\delta}(v)\bar{\psi}(\bar{Y}_r) - \bar{\phi}_{\delta}''(v)\int_{0}^r\bar{\psi}(\bar{Y}_a)da.
\end{align}

\begin{lemma} \label{lemma2018-01-02-6}
For any $p \geq 2$, there exists a constant $c_p$, not depending on $y_0 \in  J$, such that for all $\delta_1, \, \delta_2 > 0$,
\begin{align}\label{eq2018-01-04-10}
\|\delta(u_{\bar{A}})\|_{L^p(\Omega)} \leq c_p \, \delta^{3/2}.
\end{align}
\end{lemma}
\begin{proof}
The proof is similar to that of Lemma \ref{lemma05-04-1}.
Since $u_{\bar{A}}$ is adapted, by Proposition \ref{lemma2017-05-04-2} and \eqref{eq2018-10-09-2}, we have
\begin{align} \label{eq2017-05-04-300}
\delta(u_{\bar{A}}) &= \int_{0}^{\Delta_{\bullet}}\int_0^1\bar{\phi}_{\delta}(v)\bar{\psi}(\bar{Y}_r)W(dr, dv) - \int_{0}^{\Delta_{\bullet}}\int_0^1W(dr, dv)\bar{\phi}_{\delta}''(v)\int_{0}^r\bar{\psi}(\bar{Y}_a)da.
\end{align}
For the first term on the right-hand side of \eqref{eq2017-05-04-300}, by Burkholder's inequality, for any $p \geq 2$,  since $0 \leq \bar{\psi} \leq 1$,
\begin{align}\label{eq2017-05-04-400}
\left|\left|\int_{0}^{\Delta_{\bullet}}\int_0^1\bar{\phi}_{\delta}(v)\bar{\psi}(\bar{Y}_r)W(dr, dv)\right|\right|_{L^p(\Omega)}^p
&  \leq c_p\mbox{E}\left[\left(\int_{0}^{\Delta_{\bullet}}\int_0^1\bar{\phi}_{\delta}^2(v)\bar{\psi}^2(\bar{Y}_r)drdv\right)^{p/2}\right]\nonumber\\
 &  \leq c_p \Delta_{\bullet}^{p/2}\left(\int_0^1\bar{\phi}_{\delta}^2(v)dv\right)^{p/2}\nonumber\\
&  \leq c_p \Delta_{\bullet}^{p/2}\delta^{p/2} =  c_p \delta^{3p/2}.
\end{align}
For the second term on the right-hand side of \eqref{eq2017-05-04-300}, similarly, by Burkholder's inequality, for any $p \geq 2$, since $0 \leq \bar{\psi} \leq 1$,
\begin{align}\label{eq2017-05-04-500}
& \left|\left|\int_{0}^{\Delta_{\bullet}}\int_0^1W(dr, dv) \, \bar{\phi}_{\delta}''(v)\int_{0}^r\bar{\psi}(\bar{Y}_a)da\right|\right|_{L^p(\Omega)}^p \nonumber \\
  & \quad \leq c_p\mbox{E}\left[\left(\int_{0}^{\Delta_{\bullet}}dr\int_0^1dv \, (\bar{\phi}_{\delta}''(v))^2\left(\int_{0}^r\bar{\psi}(\bar{Y}_a)da\right)^2\right)^{p/2}\right]\nonumber\\
& \quad \leq c_p \left(\int_{0}^{\Delta_{\bullet}}r^2dr\right)^{p/2}\left(\int_0^1(\bar{\phi}_{\delta}''(v))^2dv\right)^{p/2}\nonumber\\
& \quad \leq c_p \Delta_{\bullet}^{3p/2}\left(\int_{y_0 - \delta}^{y_0 + 2\delta}\delta^{-4}dv\right)^{p/2}  =  c'_p \Delta_{\bullet}^{3p/2}\delta^{-3p/2} =  c'_p \delta^{3p/2},
\end{align}
where in the third inequality we use \eqref{eq2017-11-26-4000}.
Hence \eqref{eq2017-05-04-300}, \eqref{eq2017-05-04-400} and \eqref{eq2017-05-04-500} prove the lemma.
\end{proof}

\begin{lemma}
There exists a constant $c$, not depending on $y_0 \in  J$, such that for all $\delta_1, \, \delta_2 > 0$,
\begin{align}\label{eq2018-01-04-11}
\|u_{\bar{A}}\|_{\mathscr{H}} \leq c \, \delta^{3/2}.
\end{align}
\end{lemma}
\begin{proof}
The proof is similar to that of \eqref{eq2017-13}. By the definition of $u_{\bar{A}}$,
\begin{align}\label{eq2017-1300}
\|u_{\bar{A}}\|_{\mathscr{H}}^2 &\leq 2\int_{0}^{\Delta_{\bullet}}dr\int_0^1dv \, \bar{\psi}(\bar{Y}_r)^2\bar{\phi}_{\delta}^2(v) + 2\int_{0}^{\Delta_{\bullet}}dr\int_0^1dv \, (\bar{\phi}_{\delta}''(v))^2\left(\int_{0}^r\bar{\psi}(\bar{Y}_a)da\right)^2 \nonumber\\
&\leq 2\Delta_{\bullet} \int_{y_0 - \delta}^{y_0 + 2\delta}dv  + 2c\int_{0}^{\Delta_{\bullet}}r^2dr\int_{y_0 - \delta}^{y_0 + 2\delta}\delta^{-4}dv \nonumber\\
&= c \, \delta^{3} + c \, \Delta_{\bullet}^{3}\delta^{-3}  = 2c \, \delta^{3},
\end{align}
where, in the second inequality, we use \eqref{eq2017-11-26-4000}.
\end{proof}

\begin{lemma} \label{lemma2018-01-02-8}
For any $p \geq 2$, there exists a constant $c_p$, not depending on $y_0 \in  J$, such that for all $\delta_1, \, \delta_2 > 0$,
\begin{align} \label{eq2018-01-04-12}
\|\langle D\gamma_{\bar{A}}, u_{\bar{A}} \rangle_{\mathscr{H}}\|_{L^p(\Omega)} \leq c \, \delta^{7/2}.
\end{align}
\end{lemma}

\begin{proof}
The proof is similar to that of Lemma \ref{lemma2017-08-29-1}.
Taking the Malliavin derivative of $\gamma_{\bar{A}}$, we have
\begin{align*}
\langle D\gamma_{\bar{A}}, u_{\bar{A}} \rangle_{\mathscr{H}} = \int_{0}^{\Delta_{\bullet}}\bar{\psi}'(\bar{Y}_r)\langle D\bar{Y}_r, u_{\bar{A}} \rangle_{\mathscr{H}}dr.
\end{align*}
By H\"{o}lder's inequality,  \eqref{eq2017-11-26-100} and \eqref{eq2017-1300}, for any $p \geq 1$,
\begin{align*}
\mbox{E}\left[|\langle D\gamma_{\bar{A}}, u_{\bar{A}} \rangle_{\mathscr{H}}|^p\right] &\leq \|\bar{\psi}'\|^p_{\infty}\Delta_{\bullet}^{p - 1}\int_{0}^{\Delta_{\bullet}}\mbox{E}[|\langle D\bar{Y}_r, u_{\bar{A}} \rangle_{\mathscr{H}}|^p]dr \nonumber \\
&\leq c_p\bar{R}^{-p}\Delta_{\bullet}^{p -1}\int_{0}^{\Delta_{\bullet}}\mbox{E}[\|D\bar{Y}_r\|^p_{\mathscr{H}}\|u_{\bar{A}}\|_{\mathscr{H}}^p]dr \nonumber \\
&\leq  c_p\bar{R}^{-p}\Delta_{\bullet}^{p - 1 + 3p/4}\int_{0}^{\Delta_{\bullet}}\mbox{E}[\|D\bar{Y}_r\|^p_{\mathscr{H}}]dr. \nonumber
\end{align*}
Applying \eqref{eq2018-01-03-10}, this is bounded above by
\begin{align*}
& c_p\bar{R}^{-p}\Delta_{\bullet}^{p - 1 + 3p/4}\delta^{(p_0 - \gamma_0)q}\int_{0}^{\Delta_{\bullet}}r^{2p}dr \nonumber \\
&\quad = c_p\bar{R}^{-p}\Delta_{\bullet}^{p - 1 + 3p/4}\delta^{(p_0 - \gamma_0)p}\Delta_{\bullet}^{2p + 1}\leq c'_p \delta^{(\gamma_0 - p_0 - 4)p}\Delta_{\bullet}^{p - 1 + 3p/4}\delta^{(p_0 - \gamma_0)p}\Delta_{\bullet}^{2p + 1} = c'_p \delta^{7p/2},
\end{align*}
where, in the inequality, we use \eqref{eq2017-12-1-30011}.
\end{proof}

\begin{proof}[Proof of Proposition \ref{prop0511-100}(a)]
Using the property of Skorohod integral $\delta$ (see \cite[(1.48)]{Nua06}),
\begin{align}\label{eq2018-01-04-5}
\delta(u_{\bar{A}}/\gamma_{\bar{A}}) &= \frac{\delta(u_{\bar{A}})}{\gamma_{\bar{A}}} + \frac{\langle D\gamma_{\bar{A}}, u_{\bar{A}} \rangle_{\mathscr{H}}}{\gamma_{\bar{A}}^2}  := I_1 + I_2.
\end{align}
By Lemmas \ref{lemma2018-01-02-6} and \ref{lemma2017-05-12-1}(b),
\begin{align}\label{eq2018-01-04-6}
\|I_1\|_{L^2(\Omega)} \leq c \, \delta^{3/2} \delta^{-2} = c \, \delta^{-1/2}.
\end{align}
By Lemmas \ref{lemma2018-01-02-8} and \ref{lemma2017-05-12-1}(b),
\begin{align}\label{eq2018-01-04-7}
\|I_2\|_{L^2(\Omega)} \leq c \, \delta^{7/2} \delta^{-4} = c \, \delta^{-1/2}.
\end{align}
Therefore, \eqref{eq2018-01-04-5}, \eqref{eq2018-01-04-6} and \eqref{eq2018-01-04-7} establish \eqref{eq2018-01-02-17}.
\end{proof}

\subsection{Proof of Proposition \ref{prop0511-100}(b)}
\begin{proof}[Proof of Proposition \ref{prop0511-100}(b)]
The proof is similar to that of \eqref{eq2016-12-27-20000}.
We denote
\begin{align*}
\sigma_0^2 := \sup\limits_{(t, x) \in [0, \delta_1] \times [y_0, y_0 + \delta_2]}\mbox{E}[u(t, x)^2].
 \end{align*}
 From \eqref{eq2016-06-14-1}, we have $\sigma_0^2 \leq C \, \delta_1^{1/2}\leq C \, (\delta_1^{1/2} + \delta_2)$. On the other hand, by \cite[(4.50)]{DKN07}, we have
\begin{align}\label{eq2017-09-150011}
\mbox{E}[M_0] &\leq \mbox{E}\left[\sup\limits_{(t, x) \in [0, \delta_1] \times [y_0, y_0 + \delta_2]}|u(t, x)|\right] \nonumber\\
&\leq \mbox{E}\left[\sup\limits_{[t^{1/2} + |x - y_0|]^{1/2}\leq (\delta_1^{1/2} + \delta_2)^{1/2}}|u(t, x)|\right] \leq c \, (\delta_1^{1/2} + \delta_2)^{1/2}.
\end{align}
Applying Borell's inequality (see \cite[(2.6)]{Adl90}), as in \eqref{eq2018-12-11-1}, we see that for all $z > c\,  (\delta_1^{1/2} + \delta_2)^{1/2}$ (here  $c$ is the constant in \eqref{eq2017-09-150011}),
\begin{align}
\mbox{P}\{M_0 > z\} &\leq 2 \exp\left(-(z - \mbox{E}[M_0])^2/(2\sigma_0^2)\right)\leq \bar{c}\exp\left(-z^2/(3C(\delta_1^{1/2} + \delta_2))\right).
\end{align}
Since for $0 \leq z \leq c\,  (\delta_1^{1/2} + \delta_2)^{1/2}$,
$
\exp(-z^2/(3C(\delta_1^{1/2} + \delta_2))) \geq e^{-c^2/(3C)},
$
we can find a constant $\tilde{c}$ such that for all $z > 0$,
\begin{align*}
\mbox{P}\{M_0 > z\} \leq \tilde{c}\exp(-z^2/(3C(\delta_1^{1/2} + \delta_2))).
\end{align*}
This proves \eqref{eq2018-01-4-3}.
\end{proof}

\begin{remark}
The results of Theorems \ref{the2017-11-25-1}, \ref{theorem100}, and \ref{theorem2018-01-03-1} also hold for the solution of \eqref{eq2017-10-11-62} without boundary ($x \in \mathbb{R}$). This is because in the definition of the random variables $H$ and $\bar{H}$ in (\ref{eq2017-05-31-10000}) and \eqref{eq2017-05-31-1000011}, the functions $\phi_{\delta_1}$  and $\bar{\phi}_{\delta}$ are compactly supported and $C^{\infty}$ and the boundary conditions do not affect the smoothness of $u_A^i, \gamma_A^{i, j}, i, j \in \{1, 2\}$ and $u_{\bar{A}}, \, \gamma_{\bar{A}}$ in Lemma \ref{lemma2017-1}. And the equalities (\ref{eq2017-08-28-1}), (\ref{eq2017-08-28-2}), (\ref{eq2017-08-28-3}), (\ref{eq2017-08-28-4}) and \eqref{eq2018-01-02-13} in the proof of Lemma \ref{lemma2017-2} still hold with $[0, 1]$ replaced by $\mathbb{R}$. Furthermore, the formulas and estimates in (\ref{eq2017-2}), \eqref{eq2018-01-02-15}, (\ref{eq2017-09-07-10}), (\ref{eq2017-09-07-13}), (\ref{eq2017-08-28-10}), (\ref{eq2017-09-07-14}) and \eqref{eq2018-01-04-5} are generic, no matter with or without boundary conditions.
Moreover, the boundary conditions do not change the estimates in (\ref{eq2017-6}), (\ref{eq2017-10}), (\ref{eq2017-13}), (\ref{eq2017-05-11-3}), (\ref{eq2017-05-11-4}), \eqref{eq2018-01-04-10}, \eqref{eq2018-01-04-11} and \eqref{eq2018-01-04-12}. Furthermore, the equalities  (\ref{eq2017-08-31-4}), (\ref{eq2017-05-11-6}), (\ref{eq2017-05-11-8}), (\ref{eq2017-05-11-100}) and  (\ref{eq2017-05-11-9}) remain the same. In the end, the estimates (\ref{eq2017-09-15}) and \eqref{eq2017-09-150011} still hold for the solution of the equation without boundary (we can redo the proof of \cite[(4.50)]{DKN07} line by line).
\end{remark}

\begin{small}

\vspace{1.5cm}

\noindent\textbf{Robert C. Dalang} and \textbf{Fei Pu.}
Institut de Math\'ematiques, Ecole Polytechnique
F\'ed\'erale de Lausanne, Station 8, CH-1015 Lausanne,
Switzerland.\\
Emails: \texttt{robert.dalang@epfl.ch} and \texttt{fei.pu@epfl.ch}\\
\end{small}
\end{document}